\documentclass[twosided,11pt]{article}
\usepackage[preprint,nohyperref]{jmlr2e}
\usepackage{tikz}
\usetikzlibrary{3d,calc}
\usepackage{mathrsfs,placeins,soulutf8}
 \usepackage[T1]{fontenc}
\usepackage[utf8]{inputenc}
\usepackage{geometry}
\usepackage[linesnumbered,ruled]{algorithm2e} 
\usepackage{tikz} 
\usepackage{mathtools}
\usetikzlibrary{shapes.geometric, arrows}
\usepackage{algorithmic}
\usepackage{bm}
\usepackage{booktabs}
\usepackage[english]{babel}
\usepackage{adjustbox}
\usepackage{subcaption,afterpage} 
\usepackage{mathtools}
\usepackage{amsfonts}
\usepackage[colorinlistoftodos]{todonotes}
\usepackage{enumerate}
\usepackage{hyperref}
\newcommand{\excise}[1]{}

\usepackage{url}
\usepackage{bm}
\usepackage{multirow}
\usepackage{array}
\usepackage{amssymb}
\usepackage{graphicx}
\usepackage{wasysym}

\makeatletter

\newcommand{\lyxdot}{.}

\newtheorem{thm}{\protect\theoremname}
\newtheorem{lem}[thm]{\protect\lemmaname}
\newtheorem{rem}[thm]{\protect\remarkname}
\newtheorem{cor}[thm]{\protect\corollaryname}
\newtheorem{defn}[thm]{\protect\definitionname}
\newtheorem{prop}[thm]{\protect\propositionname}
\@ifundefined{date}{}{\date{}}
\usepackage{babel}
\usepackage{adjustbox}
\usepackage{soul}
\usepackage{placeins}
\usepackage{multirow}

\providecommand{\corollaryname}{Corollary}
\providecommand{\definitionname}{Definition}
\providecommand{\examplename}{Example}
\providecommand{\lemmaname}{Lemma}
\providecommand{\propositionname}{Proposition}
\providecommand{\remarkname}{Remark}
\providecommand{\theoremname}{Theorem}

\usepackage{comment}
\usepackage{lastpage}
\jmlrheading{25}{2024}{1-\pageref{LastPage}}{4/23; Revised
3/24}{6/24}{23-0547}{Hengrui Luo and Meng Li}

\ShortHeadings{Ranking Perspective for Tree-based Symbolic Regressions}{Luo and Li}
\firstpageno{1}

\begin{document}

\title{Ranking Perspective for Tree-based Methods with Applications to Symbolic Feature Selection}

\author{\name Hengrui Luo \email hrluo@rice.edu\\
		\addr   Lawrence Berkeley National Laboratory\\
		Berkeley, CA, 94720, USA \\
        Department of Statistics, Rice University\\
        Houston, TX, 77005, USA 
  \AND
     \name  Meng Li \email meng@rice.edu \\
       \addr Department of Statistics\\
       Department of Statistics, Rice University\\
       Houston, TX 77005, USA 
}

\editor{}

\maketitle
\begin{abstract}
Tree-based methods are powerful nonparametric techniques in statistics and machine learning. However, their effectiveness, particularly in finite-sample settings, is not fully understood. Recent applications have revealed their surprising ability to distinguish transformations (which we call \textit{symbolic feature selection}) that remain obscure under current theoretical understanding. This work provides a finite-sample analysis of tree-based methods from a ranking perspective. We link oracle partitions in tree methods to response rankings at local splits, offering new insights into their finite-sample behavior in regression and feature selection tasks. Building on this local ranking perspective, we extend our analysis in two ways: (i) We examine the global ranking performance of individual trees and ensembles, including Classification and Regression Trees (CART) and Bayesian Additive Regression Trees (BART), providing finite-sample oracle bounds, ranking consistency, and posterior contraction results. (ii) Inspired by the ranking perspective, we propose concordant divergence statistics $\mathcal{T}_0$ to evaluate symbolic feature mappings and establish their properties. Numerical experiments demonstrate the competitive performance of these statistics in symbolic feature selection tasks compared to existing methods.

\end{abstract}

\section{\label{sec:Introduction}Introduction}

\subsection{Tree regressions}

Tree-based methods, including CART \citep{breiman2017classification, agarwal2022hierarchical}, Bayesian CART \citep{chipman1998bayesian}, and their ensembles
such as random forests \citep{hastie2009elements} and Bayesian additive regression trees (BART \citet{chipman_bart_2010}), are popular nonparametric techniques
in statistics and machine learning \citep{breiman2017classification,hastie2009elements}. These tree-based methods
are highly effective in practice, showing competitive performance
in wide-ranging tasks such as regression \citep{grinsztajn2022tree,hrluo_2022e,LHL2024},
classification \citep{zhu2020classification}, causal inference
\citep{hahn2020bayesian}, and feature selection \citep{bleich_variable_2014}. Their empirical
success has motivated a growing literature that aims to provide theoretical
guarantees \citep{linero2018bayesian,athey2019generalized,rovckova2020posterior,ronen2022mixing}.
However, much of this research relies on asymptotic analysis, and
finite-sample understanding, which directly addresses observed empirical
effectiveness, remains limited.

In addition, some of their recent applications in interpretable machine
learning show success beyond the reach of current theoretic understanding. In
particular, \citet{ye2021operator} has proposed the use of BART to achieve
effective and scalable regression tools for symbolic regression, a
rapidly developing field that seeks to identify the nonlinear dependence
between data from a given set of mathematical expressions \citep{makke2024interpretable}. The authors
show empirical evidence that BART is able to distinguish transformations
of the same active variable with a small sample size, which is a crucial
property to ensure accurate performance of a nonparametric variable
selection method in symbolic regression. The existing statistical
literature has a limited scope to decipher this property as the relevance
of transformations of the same active variable is invariant from a
traditional nonparametric variable selection perspective. 

In this article, we attempt to offer new insights from ranking perspective in the current paper and propose a new divergence based on this perspective. We address the challenges above by connecting tree-based
methods with ranking \citep{clemenccon2015treerank,clemenccon2019tree}, which is an underdeveloped perspective in the literature on feature selection. This connection is inherently within the finite-sample
regime. It leads to conceptual elucidation of a class of tree-based
methods, and additionally motivates new statistics that are particularly
useful for screening transformations in symbolic regression. Our finite-sample
connection is broadly applicable to Bayesian and non-Bayesian tree-based
methods. We also establish asymptotic theory by connecting BART with
ranking.

Following the background introduced in Section \ref{sec:Introduction}, the
rest of the paper is organized as follows: Section \ref{sec:Problem-setting} introduces the notations, concept of trees, and the principal decision ratios, around which we organized our discussions. Section \ref{sec:Local-ranking-at} analyzes the oracle partition, 
refines  per-node analysis along a single tree and discusses the
feature selection behavior when we only consider local splits on univariate
input variables. Section \ref{sec:Global-rankings-with} furthers our
discussions from the behavior of local splits to the behavior of global
regression using CART and BART with a focus on ranking. Section \ref{subsec:Concordant--Statistics} introduces
a novel divergence statistics $\mathcal{T}_{0}$ inspired by the local splits studied. 
Section \ref{sec:Experiments}
provides experimental evidence showing that our theoretical
analysis and $\mathcal{T}_{0}$ results match the BART symbolic feature selection
on synthetic datasets. Section \ref{sec:Future-work} concludes the paper and
outlines several directions for future research.

\subsection{Tree methods for variable selections}

Variable importance is crucial for identifying significant features.
Tree-based methods, such as CART and BART, have significantly advanced
variable and feature selection \citep{linero2018bayesian,breiman2017classification,bleich_variable_2014}.
They provide robust mechanisms for assessing variable importance and
operational selection. Random forests are also widely used for this purpose through various mechanisms.
A foundational work applied random forests to microarray data for
gene selection, demonstrating the method's accuracy in identifying
relevant genes, outperforming other techniques \citep{diaz2006gene}.

On one hand, random forests can produce variable importance indices for variable and model selections \citep{genuer2010variable}. For example, the ranger implementation optimized random forests for
high-dimensional data and large-scale applications \citep{wright2017ranger},
which provides measures of variable importance from the random forest.
These measures indicate the significance of each feature in predicting
the response variable. To address biases in random forests especially
with correlated predictors \citep{strobl2008conditional}, instead
of permuting values unconditionally, the values are permuted conditionally
based on the values of other variables. 

Operational selection can be achieved during dynamic growth of
the tree and involves actively choosing variables based on specific
criteria. For example, the Boruta method \citep{kursa2010feature} uses random
forests for the selection of all relevant characteristics by comparing the original attributes
with the randomized counterparts, ensuring that all relevant characteristics are retained.
On the other hand, BART captures the 
uncertainty in variable importance, leading to more accurate models,
particularly in gene regulation studies \citep{horiguchi2021assessing}
but also induces variable importance for selection \citep{bleich_variable_2014}.

\subsection{Tree-based rankings}

\citet{clemenccon2011adaptive} explored adaptive partitioning schemes
for bipartite ranking, showing how decision trees can create adaptive
partitions to improve ranking performance. These schemes dynamically
adjust tree partitions to better capture the distribution of positive
and negative classes. Adaptive partitioning uses decision trees to
refine partitions based on ranking performance feedback. Their TreeRank
method adapts the tree structure to split 
nodes to improve the AUC of ROC. It is useful for complex data distributions
that require flexible partitioning. Furthermore, TreeRank Tournament
algorithm \citep{clemenccon2015treerank} enhances this by integrating
multiple trees, stabilizing ranking performance, and extends its capability
in feature selection.

\citet{clemenccon2009partitioning} examined partitioning rules for
bipartite ranking, highlighting the role of decision tree methods in creating
effective partitions to approximate the optimal ROC curve. They showed
that tree-based partitioning rules could be optimized to improve ranking
by focusing on informative splits that maximize class separation.
CART variants and other decision tree methods are essential for statistical
ranking problems, including bipartite ranking \citep{menon2016bipartite,uematsu2017theoretically}.

Beyond bipartite ranking, a scoring function can be estimated from leaf nodes of a CART or other regression model \citep{cossock2006subset}.
When an instance is classified into a leaf node, the score of that
leaf node serves as the estimated score for the instance. This approach
transforms the tree model into a scoring function that can be used
for ranking purposes, which we will revisit in Section \ref{sec:Global-rankings-with}. 
\global\long\def\mY{\mathcal{Y}}%

\section{\label{sec:Problem-setting}Oracle Partitions}

\subsection{Notation}

We consider the regression problem and the following
data generating model: 
\begin{align}
y_{i} & =f(\bm{x}_{i})+\epsilon_{i},\quad i=1,2,\cdots,N,\label{eq:noisy score-1}\\
\epsilon_{i} & \sim N(0,\sigma^{2}_{\epsilon})\nonumber \\
\bm{z}_{i} & =\left(\theta_{1}\bm{x}_{i},\cdots,\theta_{q}\bm{x}_{i}\right)\in\mathbb{R}^{q},\label{eq:feature_def}
\end{align}
with continuous covariates variables $\bm{x}_{i}\in\mathbb{R}^{d}$
and their transformed symbolic features $\bm{z}_{i}\in\mathbb{R}^{q}$ under
a collection of feature mappings $\theta_{\bullet}:\mathbb{R}^{d}\rightarrow\mathbb{R}$.
This collection of feature mappings $\theta_{1},\cdots,\theta_{q}$
is usually constructed and selected by symbolic regressions (e.g.,
$\sin,\cos,\exp$), to approximate continuous responses $y_{i}\in\mathbb{R}$.
When $q=d$ and $\theta_{i}=\pi_{i}$ (i.e., projection onto the $i$-th
coordinate) \eqref{eq:noisy score-1} reduces to the usual regression
setting of $y_{i}=f(\bm{x}_{i})+\epsilon_{i}$. In other words, the
feature mappings $\theta_{1},\cdots,\theta_{q}$ allow us to consider
a more general regression setting by sending the data matrix into
the following feature matrix 
\begin{align*}
\left(\begin{array}{c}
\bm{x}_{1}\\
\vdots\\
\bm{x}_{N}
\end{array}\right)\in\mathbb{R}^{N\times d}\mapsto\left(\begin{array}{ccc}
\theta_{1}\bm{x}_{1}, & \cdots & ,\theta_{q}\bm{x}_{N}\\
\vdots &  & \vdots\\
\theta_{1}\bm{x}_{N}, & \cdots & ,\theta_{q}\bm{x}_{N}
\end{array}\right)=\left(\begin{array}{c}
\bm{z}_{1}\\
\vdots\\
\bm{z}_{N}
\end{array}\right) & \in\mathbb{R}^{N\times q}.
\end{align*}
We shall distinguish $\bm{x}_i$ and $\bm{z}_i$ by referring to $\bm{x}_i$ as \textit{inputs} and $\bm{z}_i$ as \textit{features} or \textit{predictors}. The population counterpart of the model above omits the index $i$ in the notation when we do not need to refer to each sample. 
\begin{figure}[t]
\centering

\includegraphics[width=0.95\textwidth]{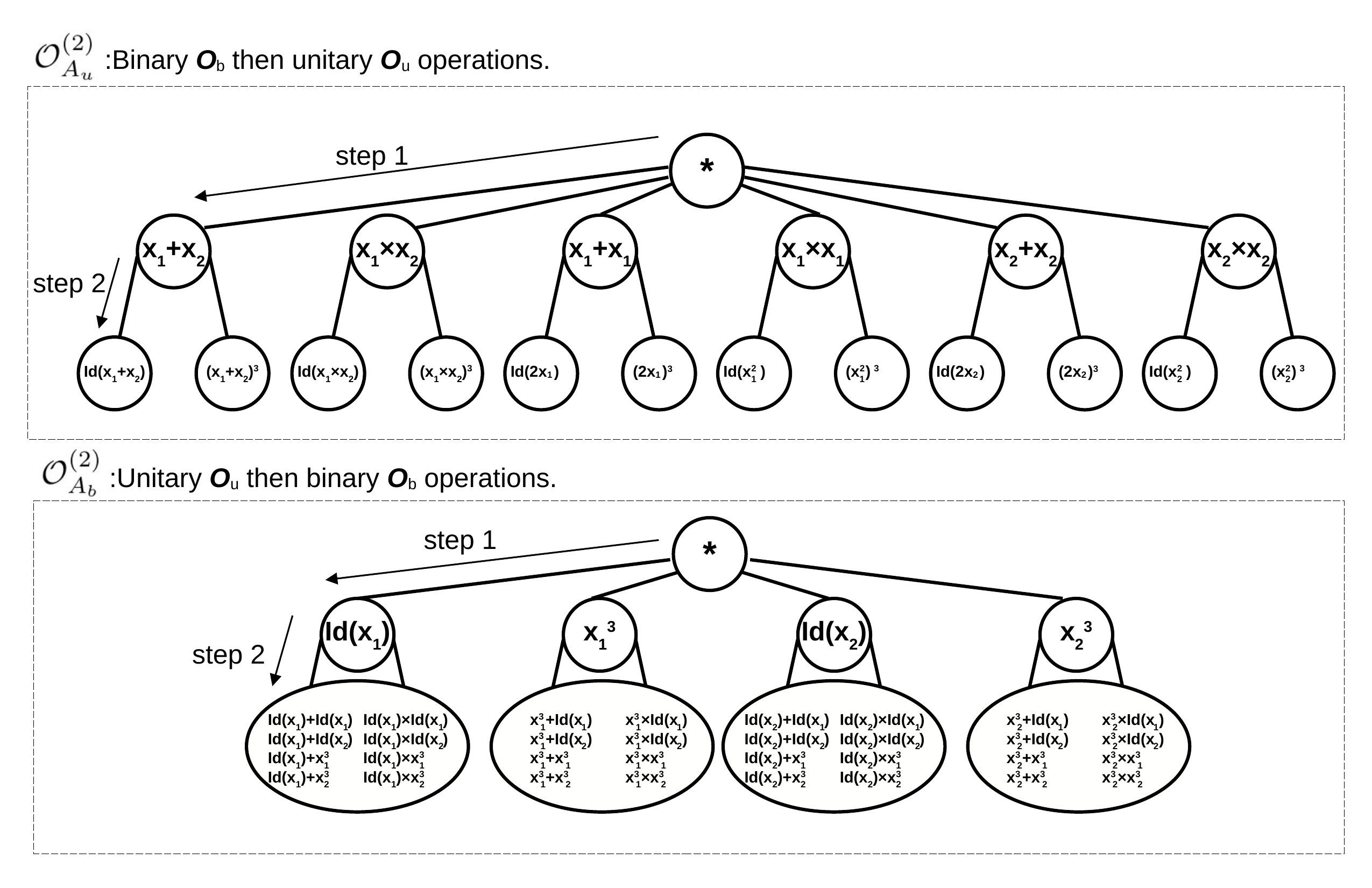}

\caption{\label{fig:1layer} We illustrate 2-layer symbolic regression with
$\mathcal{O}_{u}=\{id,x^{3}\}$ and $\mathcal{O}_{b}=\{+,\times\}$.
We also follow the notation convention $\mathcal{O}_{A_{u}}^{(2)}$
and $\mathcal{O}_{A_{b}}^{(2)}$ for the architectures specified in
\citet{ye2021operator}. We displayed all of the possible features in a 2-step symbolic composition using tree structure, showing the rapidly increasing number $q$ of features, namely transformed symbolic feature $\bm{z}$'s.}
\end{figure}
\begin{example}
\label{exa:(Symbolic-feature-mappings)}(Symbolic feature mappings)
Consider $d=2$ and input variable $\bm{x}=({x}_{1},{x}_{2})\in\mathbb{R}^{2}$
and an operator set $\mathcal{O}_b=\{+,\times\}$, in one step composition there
are $q=2\times2\times2=8$ features (i.e., combinations
$\left\{ {x}_{1},{x}_{2} \right\} \otimes\{+,\times\}\otimes\left\{ {x}_{1},{x}_{2}\right\} $)
${z}_{1}=2{x}_{1},{z}_{2}={x}_{1}+{x}_{2},\cdots,{z}_{8}={x}_{2}^{2}$ (but only 6 distinct features as shown in Figure \ref{fig:1layer});
suppose the next step we have an operator set $\mathcal{O}_u=\{id,x^3\}$
then in one step composition there are $q=6\times2=12$ different
features. Even with only two composition steps, the dimensionality
$q$ of features grows quickly \citep{bryant1992symbolic}. Due to
repetitive use of this composition construction in symbolic regression, $q\gg d$ and these $q$ features are usually highly
correlated, making feature (pre-)selection necessary in symbolic regressions. 

In \citet{ye2021operator}, BART is applied to perform feature selection among all these symbolic features (which are highly correlated, and identical if noiseless composition is assumed) and achieve good performance from the corresponding symbolic regression model using the chosen features. If we take different orders of compositions (i.e., taking unitary operation first, denoted as $\mathcal{O}_{A_u}^{(2)}$; or taking binary operation first, denoted as $\mathcal{O}_{A_b}^{(2)}$), the number of features will change.
\end{example}

Although we assume additive noises $\epsilon_{i}$ such that $\mathbb{E}\epsilon_{i}=0$
and $\text{Var}(\epsilon_{i})=\sigma_{\epsilon}^{2}>0$, the ranking
perspective studied in this article also covers the noise-free setting
with $\sigma_{\epsilon}^{2}=0$. We denote the pairs of observations
as 
\global\long\def\mXN{\mathcal{X}^{(N)}}%
\global\long\def\mYN{\mathcal{Y}^{(N)}}%
\global\long\def\mDN{\mathcal{D}^{(N)}}%
\global\long\def\mT{\mathcal{T}}%
\begin{align*}
\mXN & =\{\bm{x}_{1},\bm{x}_{2},\cdots,\bm{x}_{N}\}=\left\{ \left(\bm{x}_{1,1},\cdots,\bm{x}_{1,d}\right),\left(\bm{x}_{2,1},\cdots,\bm{x}_{2,d}\right),\cdots,\left(\bm{x}_{N,1},\cdots,\bm{x}_{N,d}\right)\right\} ,\\
\mathcal{Z}^{(N)} & =\{\bm{z}_{1},\bm{z}_{2},\cdots,\bm{z}_{N}\},\quad \mYN=\{y_{1},y_{2},\cdots,y_{N}\}.
\end{align*}
We use an enclosing round bracket when the observation pairs are sorted
according to the ranks of response $y$'s, i.e., 
\begin{align}
\left(\bm{z}_{(1)},y_{(1)}\right),\left(\bm{z}_{(2)},y_{(2)}\right),\cdots,\left(\bm{z}_{(N)},y_{(N)}\right),\label{eq:y_rank convention}\\
\text{ or }\left(\bm{x}_{(1)},y_{(1)}\right),\left(\bm{x}_{(2)},y_{(2)}\right),\cdots,\left(\bm{x}_{(N)},y_{(N)}\right), 
\end{align}
where $y_{(1)}<y_{(2)}<\cdots<y_{(N)};$ here $\bm{x}_{(i)}=\left(\bm{x}_{(i),1},\bm{x}_{(i),2},\cdots,\bm{x}_{(i),d}\right)$
when written in the form of each coordinate, and likewise $\bm{z}_{(i)}=\left(\bm{z}_{(i),1},\bm{z}_{(i),2},\cdots,\bm{z}_{(i),q}\right)$. %
We assume that there are no ties among the responses throughout the paper. 

\subsection{\label{subsec:Splittings in Tree-based Models}Recursive partition
in tree-based models}

A binary tree divides the predictor space and consists
of internal nodes and leaf nodes. The leaf nodes form a partition
of the predictor space; for a tree $\mathcal{T}$ consisting of $K$
leaf nodes, the conditional mean of the response $\mathbb{E}(y\mid\bm{x})=g(\bm{x})=\sum_{i=1}^{K}\mu_{i}\bm{1}\left(\left.\bm{x}\in P_{i}\right|\mathcal{T}\right),$
where $\mu_{i}\in\mathbb{R}$ is the mean value associated with the
$i$th leaf node $P_{i}$ for $i=1,\ldots,K$. 

To recursively construct binary
trees \citep{quinlan1986induction,breiman2017classification,hastie2009elements}, we can split at each node by splitting coordinates and splitting
values (a.k.a., cutoff values); in particular, for an internal node $\eta$ with split parameters
$(C_{\eta},k_{\eta})$, we divide $\eta$ into the left and right
child nodes $
\bm{x}_{\bullet,k_{\eta}}\leq C_{\eta}$ and $\bm{x}_{\bullet,k_{\eta}}>C_{\eta},$ respectively,
depending on whether the $k_{\eta}$-th coordinate is greater
than the splitting value $C_{\eta}$. With respect to every node $\eta\in\mathcal{T}$ in a tree structure, we adopt the notation $i\in\eta$ to indicate that the $i$-th sample $(\bm{x}_i,y_i)$ or $(\bm{z}_i,y_i)$ is assigned to the node $\eta$. 

\global\long\def\bx{\bm{x}}%
The recursive partition
induced by a tree method is dictated by interpretable decision rules including
the choice of splitting coordinates $k_{\eta}$ and splitting values
$C_{\eta}$. For any internal node $\eta$, let $n_{\eta}$ be its
size (i.e., the number of samples that fall into this node) and $\{y_{i}:1\leq i\leq n_{\eta}\}$
be the responses contained in $\eta$. CART selects $(C_{\eta},k_{\eta})$
by 
\begin{equation}
(C_{\eta},k_{\eta})=\underset{C,k}{\arg\min}\sum_{{i}\in\eta}(y_{i}-\mu_{L}^{C,k})^{2}\bm{1}(\bm{x}_{i,k}\leq C)+\sum_{{i}\in\eta}(y_{i}-\mu_{R}^{C,k})^{2}\bm{1}(\bm{x}_{i,k}>C),\label{eq:loss.cart}
\end{equation}
where $(\mu_{L}^{C,k},\mu_{R}^{C,k})$ are the sample average
of responses within the left and right nodes that minimizes in-node
sum of squares, i.e., 
\begin{equation}
\mu_{L}^{C,k}\coloneqq\frac{\sum_{i\in\eta}y_{i}\cdot\bm{1}\left(\bm{x}_{i,k}\leq C\right)}{\sum_{{i}\in\eta}\bm{1}\left(\bm{x}_{i,k}\leq C\right)},\quad\mu_{R}^{C,k}\coloneqq\frac{\sum_{{i}\in\eta}y_{i}\cdot\bm{1}\left(\bm{x}_{i,k}>C\right)}{\sum_{{i}\in\eta}\bm{1}\left(\bm{x}_{i,k}>C\right)}.\label{eq:mean_R_star-1}
\end{equation}
Cycling through all internal nodes induces a recursive partition defined
by the leaf nodes. In the existing tree literature \citep{quinlan1986induction,breiman2017classification,hastie2009elements},
this partition is viewed as one of the predictor space, which indeed
defines basis functions. These basis functions are used to approximate the regression function $f$ in \eqref{eq:noisy score-1}.

We take a slightly different perspective to view this partition as one that is induced by the observed
samples, noting that the splitting action separates all observations
$\mDN$ into two groups corresponding to the left and right children nodes.

Conceptually, the basis function view is at the population level operated
on the support of $f$, while our analysis is inherently finite-sample and it 
operated on the indices of observations $[N]=\{1,\ldots,N\}$. In
particular, each internal node $\eta$ collects a subset of observations
and can be characterized by a subset of $[N]$, i.e., $\eta\subset[N]$,
and the left and right child nodes of $\eta$ lead to a finer partition
of $\eta$ by $\left\{ i:\bm{x}_{i,k_{\eta}}\leq C_{\eta},i\in\eta\subset[N]\right\} $,
and $\left\{ i:\bm{x}_{i,k_{\eta}}>C_{\eta},i\in\eta\subset[N]\right\} $.
As such, the recursive partition of the prediction space, once realized
by finite samples, leads to a recursive partition of $[N]$, which
subsequently yields a recursive partition of $\mXN$ and $\mYN$.
As we shall show later, this finite-sample perspective connects tree
methods with the ranking of $\mYN$.

In the symbolic regression setting, a tree method sees the transformed
features $\mathcal{Z}^{(N)}$ and $\mathcal{Y}^{(N)}$, and would proceed using the transformed symbolic features $\mathcal{Z}^{(N)}$ instead of $\mathcal{X}^{(N)}$ as predictor. We reiterate that these transformed features are of rapidly growing dimensionality and present high correlations, which require feature (pre-)selection as a necessary step in symbolic regression. 
Then fitting a regression tree using the transformed symbolic features, the criteria \eqref{eq:mean_R_star-1}  
contains a split along the feature space $\mathcal{Z}^{(N)}$: 
\begin{align}
\mu_{L}^{C,k} & =\frac{\sum_{{i}\in\eta}y_{i}\cdot\bm{1}\left(\bm{z}_{i,k}\leq C\right)}{\sum_{{i}\in\eta}\bm{1}\left(\bm{z}_{i,k}\leq C\right)}=\frac{\sum_{{i}\in\eta}y_{i}\cdot\bm{1}\left(\theta_{k}\bm{x}_{i}\leq C\right)}{\sum_{{i}\in\eta}\bm{1}\left(\theta_{k}\bm{x}_{i}\leq C\right)},\label{eq:trans1}\\
\mu_{R}^{C,k} & =\frac{\sum_{{i}\in\eta}y_{i}\cdot\bm{1}\left(\bm{z}_{i,k}>C\right)}{\sum_{{i}\in\eta}\bm{1}\left(\bm{z}_{i,k}>C\right)}=\frac{\sum_{{i}\in\eta}y_{i}\cdot\bm{1}\left(\theta_{k}\bm{x}_{i}\leq C\right)}{\sum_{{i}\in\eta}\bm{1}\left(\theta_{k}\bm{x}_{i}\leq C\right)}.\label{eq:trans2}
\end{align}
Here, we omit the subscript and it is clear that $C=C_{\eta}$ associated with the node $\eta$.
The first key observation from \eqref{eq:trans1} and \eqref{eq:trans2}
is that by definition \eqref{eq:feature_def} the $k$-th coordinate
of $\bm{z}_{i}$ is obtained by applying the feature mapping $\theta_{k}$
to $\bm{x}_{i}$. Therefore, the split along the feature space $\mathcal{Z}^{(N)}$
can be attained by splitting along the original space $\mathcal{X}^{(N)}$.
For example, if $\theta_{k}=\pi_{1}$, which projects onto the first
coordinate, then splitting on the $k$-th coordinate of $\bm{z}_{i}$
can be equivalently obtained from splitting along $\bm{x}_{i,1}$.
The second key observation is that regardless of the growing dimensionality
of $q$, the estimators \eqref{eq:trans1} and \eqref{eq:trans2}
only rely on one splitting coordinate, making it particularly suitable for the symbolic regession scenario where $q$ grows rapidly, as shown in Example \ref{exa:(Symbolic-feature-mappings)}.

In the rest of this paper,
we focus on this feature set $\mathcal{Z}^{(N)}$, which includes the original
data $\mathcal{X}^{(N)}$ as a special case where the transformation sets map $\bm{x}_i$ to each of its coordinates such that $\bm{z}_i = \bm{x}_i$. 

\subsection{Local splits and principal decision ratio}

In the sequel, we first analyze each internal node $\eta$ and suppose
$\eta$ contains $n_{\eta}$ observations indexed by $\{i{(\eta)}:i=1,\ldots,n_{\eta}\}$
using the original serial indices. For simplicity, we drop
the dependence on $\eta$ in the node-specific notation with the understanding that
our analysis is generally applicable to any internal node $\eta$.
For example, we will use $(n,C,k)$ for $(n_{\eta},C_{\eta},k_{\eta})$,
respectively, and
with a slight abuse of notation, use $[n]=\{1,\ldots,n\}$ to denote 
the indices of the enclosed observations $\{i{(\eta)}:i=1,\ldots,n_{\eta}\}$.

At the stage of deciding the split (consisting of the splitting coordinate
and splitting value), for any two pairs of decision rules $(k_{1},C_{1})$
and $(k_{2},C_{2})$, we introduce the \textit{principal decision
ratio}:  
\begin{align}
\tau & =\frac{\exp\left(-\sum_{i=1}^{n}(y_{i}-\mu_{L}^{1})^{2}\bm{1}(\bm{z}_{i,k_{1}}\leq C_{1})-\sum_{i=1}^{n}(y_{i}-\mu_{R}^{1})^{2}\bm{1}(\bm{z}_{i,k_{1}}>C_{1})\right)}{\exp\left(-\sum_{i=1}^{n}(y_{i}-\mu_{L}^{2})^{2}\bm{1}(\bm{z}_{i,k_{2}}\leq C_{2})-\sum_{i=1}^{n}(y_{i}-\mu_{R}^{2})^{2}\bm{1}(\bm{z}_{i,k_{2}}>C_{2})\right)},\label{eq:MH_ratio_serial}
\end{align}
where $\mu_{L}^{1}=\mu_{L}^{C_{1},k_{1}}$ (and $\mu_{R}^{1}=\mu_{R}^{C_{1},k_{1}}$),
$\mu_{L}^{2}=\mu_{L}^{C_{2},k_{2}}$ (and $\mu_{R}^{2}=\mu_{R}^{C_{2},k_{2}}$)
as specified in \eqref{eq:mean_R_star-1}. This means that we select splitting
values from one of the observed input features $\bm{z}_{i}$'s,
following the common practice in \citet{breiman2017classification}.

The numerator and denominator of \eqref{eq:MH_ratio_serial} exponentiate
the loss in Equation \eqref{eq:loss.cart}; this formulation assumes
a Gaussian likelihood with unit error standard deviation, which is commonly used in model-based tree methods such as Bayesian
CART \citep{chipman1998bayesian,chipman2002bayesian}. Consequently,
$\tau$ may represent a likelihood ratio. Based on $\tau$, the splitting
coordinate $k$ and splitting value $C$ that minimize the loss
function \eqref{eq:loss.cart} would have the highest $\tau$ value
compared to any other splitting rules.

The principal decision ratio in \eqref{eq:MH_ratio_serial} also encodes
key information for Bayesian trees to make splitting decisions. In
particular, the fitting procedure of Bayesian trees typically relies
on an adapted form of $\tau$. For Bayesian CART~\citep{chipman1998bayesian}
and its extension to sum-of-tree counterpart BART~\citep{chipman_bart_2010},
trees are sampled using a stochastic search via Markov chain Monte-Carlo
that utilizes a Metropolis-Hastings ratio between the original tree
and a proposed tree. And this ratio reduces to $\tau$ when comparing
two splitting decisions. Similarly, Bayesian dyadic trees~\citep{li2021learning}
would draw posterior samples of partitioning directions with probabilities
determined by all pairs of $\tau$ when the splitting values are restricted
to halves of each coordinate. The exact posterior sampling procedures
in these Bayesian tree methods also include inference on other parameters,
such as the mean and stand derivation parameters at leaf nodes, and
rules to prune a complete tree; for stochastic search, the Metropolis-Hastings
ratio also involves other operations in addition to growing a tree
via splitting internal nodes, such as the swapping operation.

Central to our analysis of tree methods is the principal decision
ratio $\tau$, and its interplay with ranking and feature selection.
As explained, the principal decision ratio \eqref{eq:MH_ratio_serial} is closely related to the loss function and Metropolis-Hastings
ratio in a per-node strategy, which is advocated by \citet{andronescu2009decision}.
Also analyzing a per-node strategy, \citet{luo2021variable} observed
that splits using relevant predictors yield higher Metropolis-Hastings
ratios, indicating a preference for these predictors in the splitting
process. This occurrence of $\tau$, which we aim to substantiate
with our analysis, exemplifies the intricate connection between feature
selection and splitting decisions. Therefore, analyzing $\tau$ provides
insight into how tree-based methods behave at local splits, and these
per-node analyses findings will later be extended to the entire decision tree
and tree ensembles.

\section{\label{sec:Local-ranking-at}Local Ranking at Local Splits}

\global\long\def\mL{\mathcal{L}}%
In this section, we will establish the connection between the response
rankings (i.e., rankings of $\mathcal{Y}^{(N)}$) and the optimal principal decision ratios in a single tree.
This leads to two observations which we summarize here at a high level: 
\begin{enumerate}
\item The optimal partition of samples into children nodes depends only
on the rankings of $y$, which we called the \emph{oracle partition}.
This oracle partition is solely determined by the ranks (or orders) of response $y$, in the sequential minimizations
of the $L_{2}$ loss function~\eqref{eq:loss.cart} in splitting
parameters $C,k$ (where $C$ can only be selected from samples), when
the sizes of the nodes of the children are fixed. When children node sizes are not fixed,
the minimization of $L_{2}$ loss involves the relative magnitude
of the responses $y$. 
\item However, the partitions in tree-based methods are determined by the
actual configuration of predictors $\bm{z}$'s. This means that
even if the oracle partition is completely determined by the responses
$y$, we may only attain sub-optimal partitions in the form of $\{(\bm{z}_{i},y_{i})\mid\bm{z}_{i,k}\leq C\}$
and $\{(\bm{z}_{i},y_{i})\mid\bm{z}_{i,k}>C\}$ for some $k$ and $C$.
\end{enumerate}
We rewrite the loss function in \eqref{eq:loss.cart} in a more general
form as a function of partitions $P_{1},P_{2}$ 
\begin{equation}
\mL(P_{1},P_{2})=SS(P_{1})+SS(P_{2})=:\sum_{i:y_{i}\in P_{1}}(y_{i}-\mu_{1}^{*})^{2}+\sum_{i:y_{i}\in P_{2}}(y_{i}-\mu_{2}^{*})^{2},\label{eq:loss.y.ranking}
\end{equation}
where $(P_{1},P_{2})$ is a 2-partition of $\{y_{1},\ldots,y_{n}\}$,
and the internal sum of squares $SS(P)\coloneqq\sum_{i:y_{i}\in P}(y_{i}-\mu(P))^{2}$
with $\mu(P)=\frac{1}{|P|}\sum_{i:y_{i}\in P}y_{i}$. Then the loss
function in \eqref{eq:loss.cart} is $\mL(P_{1},P_{2})$ when $(P_{1},P_{2})$
are constrained to take the form $P_{1}=\{y_{i}:\bm{z}_{i,k}\leq C\}$
and $P_{2}=\{y_{i}:\bm{z}_{i,k}>C\}$, and $\mu_{1}^{*},\mu_{2}^{*}$ are the corresponding means of $y_i$'s associated with these two partitions. This function $\mL(P_{1},P_{2})$ is invariant
to the ordering of $(P_{1},P_{2})$, and its optimal solution
is expected to be unique only up to this ordering.

\subsection{\label{subsec:Response-rankings-with}Optimal 2-partition}

We next study the optimal partition without imposing the aforementioned predictor-dependent
constraints on the partition induced by a tree method. This study is applicable to any selected
coordinate $k\in[q]$. We begin with an illustrative example that
contains 5 observations where we assume the dimensionality $q$ of $\mathcal{Z}^{(N)}$ to be 1 for the ease of visualization in Figure \ref{fig:depth-2 regression tree-1} but the argument works regardless of $q$. 
\begin{example}
\label{exa:(5-sample-univariate-example)}(5-sample example oracle
2-partition with fixed partition sizes) Assume $n=5$ and consider
2-partitions $(P_{1},P_{2})$ with sizes $(2,3)$. We proceed with the 
explicit computation of the loss function~\eqref{eq:loss.y.ranking}.
First consider the 2-partition with $P_{1}=\{y_{(1)},y_{(3)}\}$ and
$P_{2}=\{y_{(2)},y_{(4)},y_{(5)}\}$: 
\begin{align*}
SS(P_{1}) & =\left(y_{(1)}-\frac{y_{(1)}+y_{(3)}}{2}\right)^{2}+\left(y_{(3)}-\frac{y_{(1)}+y_{(3)}}{2}\right)^{2}=\frac{1}{2}\left(y_{(1)}-y_{(3)}\right)^{2},
\end{align*}
\begin{align*}
SS(P_{2}) & =\left(y_{(2)}-\frac{y_{(2)}+y_{(4)}+y_{(5)}}{3}\right)^{2}+\left(y_{(4)}-\frac{y_{(2)}+y_{(4)}+y_{(5)}}{3}\right)^{2}+\left(y_{(5)}-\frac{y_{(2)}+y_{(4)}+y_{(5)}}{3}\right)^{2}\\
 & =\left(\frac{2y_{(2)}-y_{(4)}-y_{(5)}}{3}\right)^{2}+\left(\frac{-y_{(2)}+2y_{(4)}-y_{(5)}}{3}\right)^{2}+\left(\frac{-y_{(2)}-y_{(4)}+2y_{(5)}}{3}\right)^{2}\\
 & =\frac{1}{9}\left(6y_{(2)}^{2}+6y_{(4)}^{2}+6y_{(5)}^{2}-6y_{(2)}y_{(4)}-6y_{(4)}y_{(5)}-6y_{(2)}y_{(5)}\right).
\end{align*}
Similarly, for the 2-partition with $P_{1}=\{y_{(1)},y_{(2)}\}$ and
$P_{2}=\{y_{(3)},y_{(4)},y_{(5)}\}$, we have 
\begin{align*}
SS(P_{1})=\frac{1}{2}\left(y_{(1)}-y_{(2)}\right)^{2},\quad SS(P_{2})=\frac{1}{9}\left(6y_{(3)}^{2}+6y_{(4)}^{2}+6y_{(5)}^{2}-6y_{(3)}y_{(4)}-6y_{(4)}y_{(5)}-6y_{(3)}y_{(5)}\right).
\end{align*}
Taking the difference of these two sets of expressions (note that
$y_{(1)}<y_{(2)}<y_{(3)}<y_{(4)}<y_{(5)}$) shows that the change
of $SS(P_1)$ and $SS(P_2)$, respectively denoted by $\Delta_1$ and $\Delta_2$, are both greater than zero: 
\begin{align*}
\Delta_{1} &=  \frac{1}{2}\left(y_{(1)}-y_{(3)}\right)^{2}-\frac{1}{2}\left(y_{(1)}-y_{(2)}\right)^{2} = \frac{1}{2}\left(y_{(2)}-y_{(3)}\right)\left(y_{(1)}-y_{(3)}+y_{(1)}-y_{(2)}\right) > 0, \\
\Delta_{2} &=  \frac{1}{9}\left(6y_{(2)}^{2}+6y_{(4)}^{2}+6y_{(5)}^{2}-6y_{(2)}y_{(4)}-6y_{(4)}y_{(5)}-6y_{(2)}y_{(5)}\right)\\
 & \quad\quad -\frac{1}{9}\left(6y_{(3)}^{2}+6y_{(4)}^{2}+6y_{(5)}^{2}-6y_{(3)}y_{(4)}-6y_{(4)}y_{(5)}-6y_{(3)}y_{(5)}\right)\\
 & =\frac{1}{9}\left(6y_{(2)}^{2}-6y_{(3)}^{2}-6\left(y_{(2)}-y_{(3)}\right)y_{(4)}-6\left(y_{(2)}-y_{(3)}\right)y_{(5)}\right)\\
 & =\frac{2}{3} \left(y_{(2)}-y_{(3)}\right)\left(y_{(2)}+y_{(3)}-y_{(4)}-y_{(5)}\right)>0.
\end{align*}
Therefore, the second 2-partition exchanging $y_{(2)}$ and $y_{(3)}$
reduces the total sum $\mL(P_{1},P_{2})$ in~\eqref{eq:loss.y.ranking}. Using this argument repeatedly, we will arrive
at the conclusion that $P_{1}=\{y_{(1)},y_{(2)}\}$ and $P_{2}=\{y_{(3)},y_{(4)},y_{(5)}\}$
form a (locally) optimal 2-partition with sizes (2, 3). Similarly, another
(local) optimal 2-partition with sizes (2, 3) is $P_{1}=\{y_{(4)},y_{(5)}\}$
and $P_{2}=\{y_{(1)},y_{(2)},y_{(3)}\}$. From their expressions,
we can see that only the ranks and the magnitude of the difference
between responses affect the $\Delta_{1}$ and $\Delta_{2}.$

Like expressions \eqref{eq:trans1} and \eqref{eq:trans2}, the oracle
partition depends only on the response values $y$'s. The second column
in Figure \ref{fig:depth-2 regression tree-1-lossfunc} shows two
possible configurations based on the two datasets shown in the first
column of the same figure. In both rows (a) and (b), the x-axis is
the first element in the size 2 partition component $P_{1}$, y-axis is
the second element in size 2 partition component $P_{1}$, therefore
the figure is symmetric. The minimum is attained by $P_{1}^{*}=\{y_{(1)},y_{(2)}\}=\{y_{3},y_{4}\}$
and $P_{2}^{*}=\{y_{(3)},y_{(4)},y_{(5)}\}=\{y_{1},y_{2},y_{5}\}$,
or $P_{1}^{**}=\{y_{(4)},y_{(5)}\}=\{y_{2},y_{5}\}$ and $P_{2}^{**}=\{y_{(1)},y_{(2)},y_{(3)}\}=\{y_{3},y_{4},y_{1}\}$. 
\end{example}

\begin{figure}[t]
\centering
\begin{adjustbox}{center}
\includegraphics[width=1.15\textwidth]{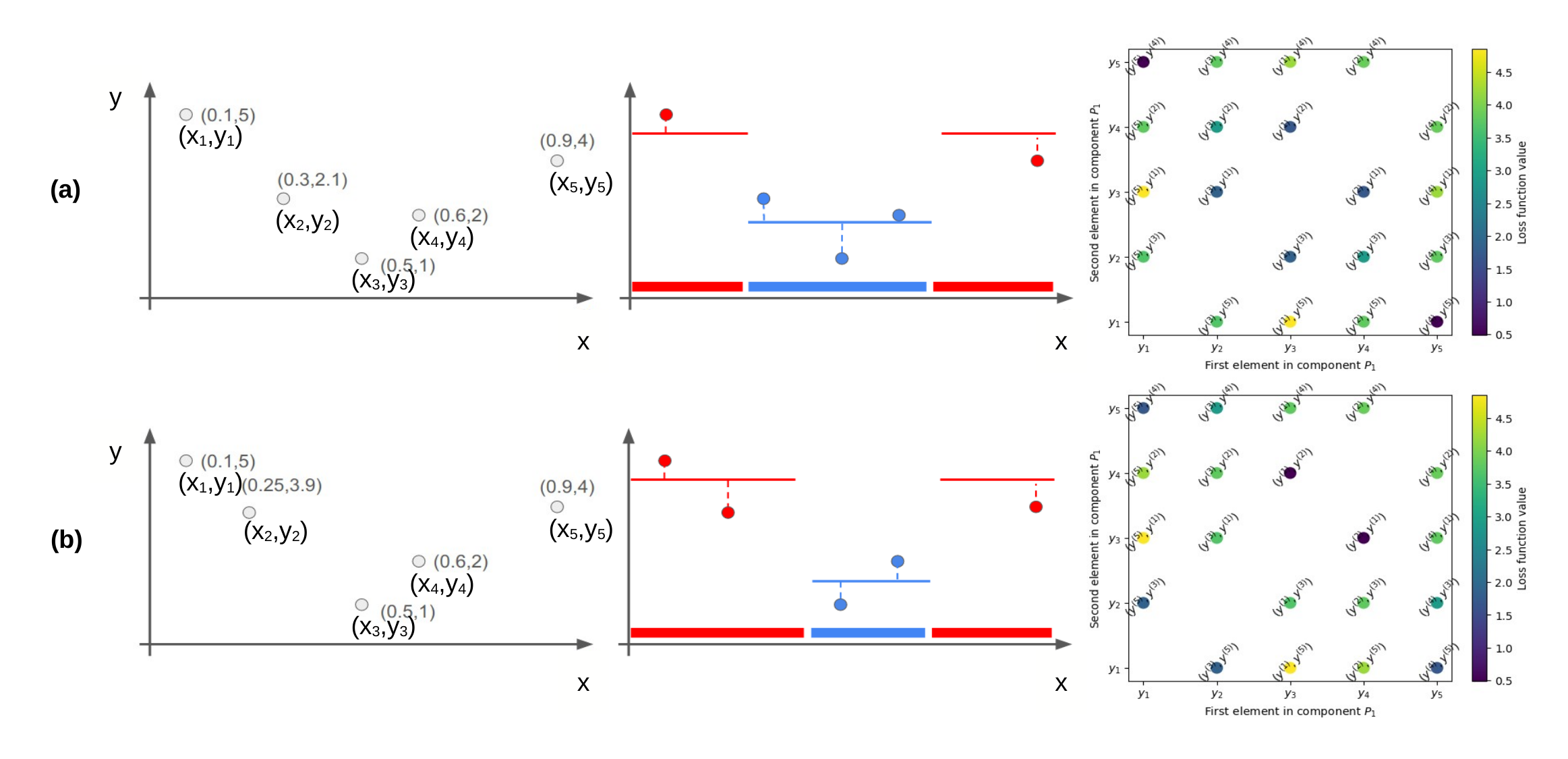}
\end{adjustbox}

\caption{\label{fig:depth-2 regression tree-1-lossfunc}\label{fig:depth-2 regression tree-1}A
depth 2 tree with 5 observations showing two possible oracle partitions
in Lemma \ref{lem:LemmaA}.\protect \protect \\
 In the first column, we present the raw $(x_{i},y_{i})$ pair of
dataset; In the second column, we present the oracle partition using
red and blue colors, and the support of indicator functions on the
$x$-axis. The horizontal solid lines represent the group mean of
$y$ values (as prediction value as well); the vertical dashed lines
represent the point-to-mean distances.\protect \protect \\
 In the third column, we illustrate the loss function \eqref{eq:loss.y.ranking}
The minimum in row (a) is attained by $\{y_{(4)},y_{(5)}\}=\{y_{1},y_{5}\}$
and $\{y_{(1)},y_{(2)},y_{(3)}\}=\{y_{2},y_{3},y_{4}\}$. The minimum
in orw (b) is attained by $\{y_{(3)},y_{(4)},y_{(5)}\}=\{y_{1},y_{2},y_{5}\}$
and $\{y_{(1)},y_{(2)}\}=\{y_{3},y_{4}\}$. We color the dots by the
actual loss function values, and annotate the ordered statistics near
each dot. }

\end{figure}

Example~\ref{exa:(5-sample-univariate-example)} serves as a representative
case for our result applicable to general sample sizes, as formalized
in the lemma below. 
\begin{lem}
\label{lem:LemmaA}(Oracle 2-partition with fixed sizes) For a 2-partition
of n elements $y_{(1)}<y_{(2)}<\cdots<y_{(n)}$ into components of
size $i$ and $n-i$, we assume that $n>4,\min(n-i,i)\geq2$ to ensure
variances are defined. Then the following partitions 
\begin{itemize}
\item $P_{1}^{*}=\{y_{(1)},y_{(2)},\cdots,y_{(i)}\}$ and $P_{2}^{*}=\{y_{(i+1)},y_{(i+2)},\cdots,y_{(n)}\}$
OR, 
\item $P_{1}^{**}=\{y_{(1)},y_{(2)},\cdots,y_{(n-i)}\}$ and $P_{2}^{**}=\{y_{(n-i+1)},y_{(i+2)},\cdots,y_{(n)}\}$ 
\end{itemize}
are the only 2-partitions of size $i$ and $n-i$ that minimize \eqref{eq:loss.y.ranking}. 
\end{lem}

\begin{proof}
See Appendix \ref{sec:Proof-of-LemmaA}. 
\end{proof}
\begin{rem}
The lemma \ref{lem:LemmaA} states that the loss function only has two
local minima attained by $(P_{1}^{*},P_{2}^{*})$ (corresponding to
size $(i,n-i)$ ) or $(P_{1}^{**},P_{2}^{**})$. Comparing the loss
function value \eqref{eq:loss.y.ranking} at $(P_{1}^{*},P_{2}^{*})$
or $(P_{1}^{**},P_{2}^{**})$ gives the global minimum of the loss
function, hence determining the split maximizing \eqref{eq:MH_ratio_serial}.
This optimal 2-partition, however, may not be attained by those 2-partitions
induced by splitting on $z$ values only once. Lemma \ref{lem:LemmaA} is a refined version of optimal splits
in trees with continuous outcome under $L^{2}$ loss, which was studied
as a grouping problem instead of a ranking problem in \citet{fisher1958grouping}. 

We next present a simple
example to demonstrate Lemma \ref{lem:LemmaA} and the induced partitions
on the response and input coordinates. 
\end{rem}

\begin{example}
\label{exa:(5-sample-univariate-example)-1}(5-sample oracle 2-partition
with varying partition sizes) Now we consider the same dataset as
in Example \ref{exa:(5-sample-univariate-example)} but we do not
fix the partition sizes to $(2,3)$ any more.

In the third column of Figure \ref{fig:depth-2 regression tree-1},
we show the oracle 2-partitions on the response $y$ and the corresponding
landscape of the loss function \eqref{eq:loss.y.ranking}. In panel (a),
we put $(x_{1},y_{1})=(0.1,5)$, $(x_{2},y_{2})=(0.3,2.1)$, $(x_{3},y_{3})=(0.5,1)$,
$(x_{4},y_{4})=(0.6,2)$, $(x_{5},x_{5})=(0.9,4)$, and the oracle
partition that minimizes the total sum is $\{y_{(1)},y_{(2)},y_{(3)}\}$
of size 3 and $\{y_{(4)},y_{(5)}\}$ of size 2. In panel (b), we put
$(x_{2},y_{2})=(0.25,3.9)$, but the rest of the pairs remain the
same, and the oracle partition that minimizes the total sum is $\{y_{(1)},y_{(2)}\}$
of size 2 and $\{y_{(3)},y_{(4)},y_{(5)}\}$ of size 3. It turns out
that the oracle 2-partition of size $(2,3)$ is indeed the configuration
that minimizes the \eqref{eq:loss.y.ranking}, compared to oracle
2-partitions of size $(1,4)$. 
\end{example}

Note that if we observe $(z_{1},y_{1}),\cdots,(z_{5},y_{5})$ instead: since the oracle partition for the loss function depends only on the response $y$'s, this does not affect our oracle partition above.
However, it is clear from Figure \ref{fig:depth-2 regression tree-1}
that each of the two candidate optimal partitions (in the sense that they
minimize \eqref{eq:loss.y.ranking}) $P_{1}^{*},P_{2}^{*}$ or $P_{1}^{**},P_{2}^{**}$
creates 3 partition components on the $\mathcal{X}$ domain, which
cannot be attained by splitting on $x$ values only once. If we choose the transformed symbolic feature $z=x^2$, then the oracle 2-partition for $y$ can be realized by partitioning on $\mathcal{Z}$ domain.
\begin{rem}
Applying Lemma \ref{lem:LemmaA} repeatedly leads to solutions to
finding optimal 2-partitions with varying sizes. In particular, for
$y_{(1)}<y_{(2)}<\cdots<y_{(n)}$ with $n>4$, the loss function \eqref{eq:loss.y.ranking}
is minimized by solving the following problem: 
\begin{align}
\underset{i\in\{1,2,\cdots,n-1\}}{\min} & \sum_{j=1}^{i}(y_{(j)}-\mu_{1}^{*})^{2}+\sum_{j=i+1}^{n}(y_{(j)}-\mu_{2}^{*})^{2}=\underset{i\in\{1,2,\cdots,n\}}{\min}\mL(P_{1},P_{2}),\label{eq:MH_min_2tree}
\end{align}
and form the associated partitions. There are still two possible minimizers
for \eqref{eq:MH_min_2tree} as stated in Lemma \ref{lem:LemmaA}.
From the angle of grouping \citep{fisher1958grouping} or analysis of variance (ANOVA), the problem of \eqref{eq:MH_min_2tree} can
be considered as finding a division of $\mathcal{Y}$ into two groups
such that the in-group variance is as small as possible. In other
words, the resulting minimizer would produce ``most in-group homogeneous''
group partitions. The optimal 2-partition for components with varying sizes
depends not only on the ranking information, as it would for the oracle
2-partition with fixed sizes, but also on the distribution information
of the responses $y$. 
\end{rem}

However, most tree models create partitions on $x$'s (e.g., which is a special 
case with $q=1,z\in\mathbb{R}$) but not $y$'s domain for prediction purposes, so how well
we can predict depends on how ``similar'' (or ``concordant'') the response rankings
and input rankings are. For example, if the inputs $x$'s and the
responses $y$'s have the same rankings, then the oracle 2-partition
on response can be realized by corresponding oracle 2-partition on
$x$'s. The following corollary \ref{cor:(Optimal-2-partition-fixed-size},
which is straightforward from Lemma \ref{lem:LemmaA}, gives a sufficient
and necessary condition to attain oracle 2-partitions when splitting
only on $x$. 
\begin{cor}
\label{cor:(Optimal-2-partition-fixed-size}(Oracle 2-partition with
fixed sizes with univariate $x$) For a 2-partition of n elements
$y_{(1)}<y_{(2)}<\cdots<y_{(n)}$ into components of fixed size $i$
and $n-i$, we assume that $n>4,\min(n-i,i)>2$ and assume that there
exists some $C$ such that 
\begin{itemize}
\item $P_{1}^{*}=\{y_{(1)}<y_{(2)}<\cdots<y_{(i)}\}=\{y'\mid\text{the pair }(x,y)\text{ s.t. }x\leq C\},$
and $P_{2}^{*}=\{y_{(i+1)}<y_{(i+2)}<\cdots<y_{(n)}\}=\{y'\mid\text{the pair }(x,y)\text{ s.t. }x>C\}.$ 
\item $P_{1}^{**}=\{y_{(1)}<y_{(2)}<\cdots<y_{(n-i)}\}=\{y'\mid\text{the pair }(x,y)\text{ s.t. }x\leq C\},$
and $P_{2}^{**}=\{y_{(n-i+1)}<y_{(n-i+2)}<\cdots<y_{(n)}\}=\{y'\mid\text{the pair }(x,y)\text{ s.t. }x>C\}.$ 
\end{itemize}
Then these are the only 2-partitions of size $i$ and $n-i$ that
minimize \eqref{eq:loss.y.ranking}. 
\end{cor}

This means that a sufficient condition for us to attain optimal
partition by splitting once on $x$ is that the ranks of $x$
and $y$ are the same or linearly related. %
In fact, when the input and response rankings are the same, the CART loss
function described by \citet{scornet2015consistency} formula (2)
or \citet{hastie2009elements} Chapter 9, takes the following form.
When input and response rankings are the same and $x_{(i)}\leq C<x_{(i+1)}$
the following loss function value remains the same: 
\begin{align*}
 \min_{i\in\{1,2,\cdots,n\}}\left(\sum_{j=1}^{i}(y_{(j)}-\mu_{1}^{*})^{2}+\sum_{j=i+1}^{n}(y_{(j)}-\mu_{2}^{*})^{2}\right)=\eqref{eq:loss.y.ranking}=\eqref{eq:MH_min_2tree}.
\end{align*}
In this case, such $x$ with the same ranking of $y$ will be the
most likely splitting coordinate. We next provide two such examples
using monotonic transformation and interpolators, respectively. 
\begin{example}
\label{cor:monotonic_f_0}(Monotonic transformation) Suppose that $d=1$ and noiseless 
$y(x)=f(x)$ is a monotonic function of univariate $x\in\mathbb{R}$,
then choosing any observation $x\in\mathcal{X}^{(N)}$ as a splitting
value will give us an oracle 2-partition corresponding to the fixed
sizes. This is because under monotonic transformation $\theta=f$, splitting
on any observed $x$ is equivalent to an oracle 2-partition on $y$.
By Lemma \ref{lem:LemmaA}, there cannot be any other 2-partition
of the same size on $y$ that gives us a strictly larger principal
decision ratio.

Furthermore, if we assume that $z=\theta_1 x$ for $d=1$ and another $\theta_1$ that is monotonic as well, then we can come to the expected conclusion
that $y(x)=f(\theta_1 x)$ always induces the optimal 2-partition, since
$\theta \circ\theta_1=f\circ\theta_1$ is again monotonic. 
This example shows that when the true underlying function is monotonic, it will induce an oracle partition on $\mathcal{X}$ and $\mathcal{Y}$ domains simultaneously.
\end{example}

\begin{example}
\label{exa:(Interpolator)-Suppose-that}(Interpolator) Suppose that $d=1$, then for any pairs
of $(x_{i},y_{i})$ for $i=1,\ldots,n$, we can construct an $n$-degree
polynomial interpolator, denoted by $g$ as a mapping, such that $y_{i}=g(x_{i})$
for all $i$. This transformed symbolic feature $z=g(x)$ maintains the same
ranking of the response and thus attains the optimal oracle 2-partition.
Then, by Corollary \ref{cor:(Optimal-2-partition-fixed-size}, using
the output of such an interpolator as input will always maximize the
\eqref{eq:loss.y.ranking} and hence the principal decision ratio.
Even without exact interpolation, when $n\ll q$, it is easy to observe
over-fitting, which means we can use the transformed features to construct such an interpolator mapping. This example shows a major difference between finite-sample and asymptotic scenarios.
\end{example}

Therefore, a decision tree tends to split on a feature for which the response
is a monotonic transformation, and likewise, an interpolator would
be the most likely splitting feature if seen by the tree. This perfectly
explains the fact that decision tree is scale-invariant in the sense
that its prediction remains unchanged if we multiply the input by
a scalar as noted in \citet{bleich_variable_2014}. In fact, this
shows a stronger result that it is multi-way scale-invariant. Namely, if we
simply multiply possibly different but non-zero scalars to each coordinate of input features,
the principal decision ratio still remains unchanged (since multiplying a non-zero scalar is a monotonic transformation and preserves the ranks). While the inability
of decision trees to distinguish between monotonic transformations
is well known in the literature \citep{bleich_variable_2014}, interpolators
are less discussed but an interesting example that shows there always
exists a mapping in finite sample sizes to mislead decision trees.
Such data-dependent mappings are typically ruled out in the pool of
symbolic expressions considered in symbolic selection, making it more robust to regular regressions.

We next analyze the principal decision ratios when selecting
between a more general class of transforms that generates symbolic features $\bm{z}$
from $\bm{x}$, other than the two examples above.

\subsection{\label{subsec:Piece-wise-monotonic-transforms}Piece-wise monotonic
transforms }

%
We now analyze the splitting behavior of the decision trees when selecting features $\bm{z}$ generated by transformations. Such selection
is crucial in symbolic regressions (See Example \ref{exa:(Symbolic-feature-mappings)}).
In particular, we consider piece-wise monotonic transforms, which
are widely used to generate symbolic features and are also of interest
due to their flexibility, as they can approximate well a sufficiently
large class of transformations of interest \citep{newman1972piecewise}.

Throughout this section, we consider two generic features $\bm{z}_{i,k_{1}}=\theta_{1}\bm{x}_{i}$
and $\bm{z}_{i,k_{2}}=\theta_{2}\bm{x}_{i}$ for an arbitrary $i$
(in shorthand, $\bm{z}_{k}=\theta\bm{x}_{k}$), where both $\theta_{1}$
and $\theta_{2}$ are piece-wise monotonic in the form of
$\theta\bm{x}_{i,k}$, where univariate mappings $\theta$ transforms the $k$-th coordinate of $\bm{x}$. 

\global\long\def\mI{\mathcal{I}}%
\global\long\def\mIr{\mathcal{I}_{1\cap2}}%
A piece-wise monotonic transform $\theta$ defined on $\mathbb{R}$
is characterized by a partition, which consists of finitely many disjoint
intervals (i.e., monotonic intervals) on each $\theta$ is monotonic.
Such partitions are not unique as one can always divide a subset while 
maintaining the strict monotonicity of $\theta$ on the finer partition.
Unless stated otherwise, we always choose the partition with the smallest
cardinality.

We aim to characterize the principal decision ratio at two pairs of
splitting values and coordinates $\left(C_{1},k_{1}\right)$ and $\left(C_{2},k_{2}\right)$,
i.e., the ratio of 
\begin{align}
\sum_{i=1}^{n}\left(y_{i}-\mu_{L}^{C_{1},k_{1}}\right)^{2}\bm{1}\left(\bm{z}_{i,k_{1}}\leq C_{1}\right)+\sum_{i=1}^{n}\left(y_{i}-\mu_{R}^{C_{1},k_{1}}\right)^{2}\bm{1}\left(\bm{z}_{i,k_{1}}>C_{1}\right)\label{eq:C1eqs}
\end{align}
and 
\begin{align}
\sum_{i=1}^{n}\left(y_{i}-\mu_{L}^{C_{2},k_{2}}\right)^{2}\bm{1}\left(\bm{z}_{i,k_{2}}\leq C_{2}\right)+\sum_{i=1}^{n}\left(y_{i}-\mu_{R}^{C_{2},k_{2}}\right)^{2}\bm{1}\left(\bm{z}_{i,k_{2}}>C_{2}\right).\label{eq:C2eqs}
\end{align}

The following definitions and simple properties of $\theta_{1}$ and
$\theta_{2}$ are useful. %

\begin{defn}
\label{def:(Refined-monotonic-intervals)}(Refined monotonic intervals)
Consider two piece-wise monotonic transforms $\theta_{1}$ and $\theta_{2}$
mapping from $\mathbb{R}$ onto $\mathbb{R}$, with monotonic intervals
$\mI_{1}$ and $\mI_{2}$ on $\mathbb{R}$ respectively. We define
the \emph{refined monotonic intervals} for the transformation pair
$(\theta_{1},\theta_{2})$ to be the collection of intervals $\mathcal{I}_{1\cap2}\coloneqq\{I=I_{1}\cap I_{2}\mid I_{1}\in\mathcal{I}_{1},I_{2}\in\mathcal{I}_{2}\}$. 
\end{defn}

The refined monotonic intervals $\mathcal{I}_{1\cap2}$ is a new partition
of the $x$-domain induced by $(\theta_{1},\theta_{2})$. On each
refined monotonic interval $I\in\mathcal{I}_{1\cap2}$, the two piece-wise
monotonic transforms $\theta_{1}$ and $\theta_{2}$ are both monotonic;
that is, the restricted transforms $\theta_{1}\mid_{I}$ and $\theta_{2}\mid_{I}$
are both monotonic. The next result follows directly from Definition
\ref{def:(Refined-monotonic-intervals)} but introduces pre-images
of the restricted transforms that are useful to study the principal
decision ratio. 
\begin{cor}
\label{cor:On-each-refined}On each refined monotonic interval $I\in\mathcal{I}_{1\cap2}$
and for any value $C\in\mathbb{R}$, the piece-wise monotonic transform
$\theta_{1}$ and $\theta_{2}$ can have 0 or 1 pre-image. That means,
there exists 0 or 1 value $t\in I$ such that the restricted transform
$\theta_{1}\mid_{I}(t)=C$ and $\theta_{2}\mid_{I}(t)=C$. 
\end{cor}

We next study the principal decision ratio of these two transformations
at arbitrary splitting values $C_{1}$ and $C_{2}$. Although $\theta_{1}$
and $\theta_{2}$ are not necessarily globally invertible, they are
invertible on each refined monotonic interval $I\in\mathcal{I}_{1\cap2}$.
By definition \eqref{eq:feature_def}, we note that $\bm{z}_{i,k_{1}}\leq C_{1}$
can be written as $\theta_{1}\bm{x}_{i,k}\leq C_{1}$, which can be
reduced to $\bm{x}_{i,k}\leq\theta_{1}^{-1}C_{1}$ on each $I\in\mathcal{I}_{1\cap2}$
where $\theta_{1}^{-1}C_{1}$ is well-defined. Using Corollary \ref{cor:On-each-refined},
for a refined interval $I\in\mathcal{I}_{1\cap2}$ we can link the
relative magnitude of principal decision ratio $\tau$ to the behavior
of covariate $\bm{x}$ instead of $\bm{z}$. The following result
Proposition \ref{prop:For-any-split-risk-decreases} is consistent
with the discussion that splitting will decrease the Bayes risk in
Section 9.3 and Theorem 9.5 of \citet{breiman2017classification}. 
\begin{prop}
\label{prop:For-any-split-risk-decreases}For any splitting value
$C\in\mathbb{R}$, the means $\mu_{L}^{C,k},\mu_{R}^{C,k}$ defined
for $\bm{z}$ as in \eqref{eq:trans2}, the parent node mean $\mu^{\sharp}$,
and the $k$-th coordinate $\bm{z}_{k}$, we have 
\begin{align*}
\sum_{i=1}^{n}(y_{i}-\mu_{L}^{C,k})^{2}\bm{1}(\bm{z}_{i,k}\leq C)+\sum_{i=1}^{n}(y_{i}-\mu_{R}^{C,k})^{2}\bm{1}(\bm{z}_{i,k}>C) & =\\
\sum_{i=1}^{n}(y_{(i)}-\mu_{L}^{C,k})^{2}\bm{1}(\bm{z}_{(i),k}\leq C)+\sum_{i=1}^{n}(y_{(i)}-\mu_{R}^{C,k})^{2}\bm{1}(\bm{z}_{(i),k}>C) & \leq\sum_{i=1}^{n}(y_{(i)}-\mu^{\sharp})^{2},
\end{align*}
where the equality holds if and only if $\mu_{L}^{C,k}=\mu_{R}^{C,k}=\mu^{\sharp}$. 
\end{prop}

\begin{proof}
See Appendix \ref{sec:Proof-of-Proposition_93}. 
\end{proof}
In the spirit of Theorem 9.5 of \citet{breiman2017classification},
we can see from the above argument that a feature mapping that induces
more splits will be favored in the sense that it increases the   principal decision ratio, hence the likelihood
for the splitted children nodes. This helps
compare two transformations in terms of the existence of pre-images, or to ``contrast'' these two transformations at a finer resolution level. Below, we start with an illustrative example, followed by its generalization.

\begin{figure}[t]
\centering

\includegraphics[width=0.99\textwidth]{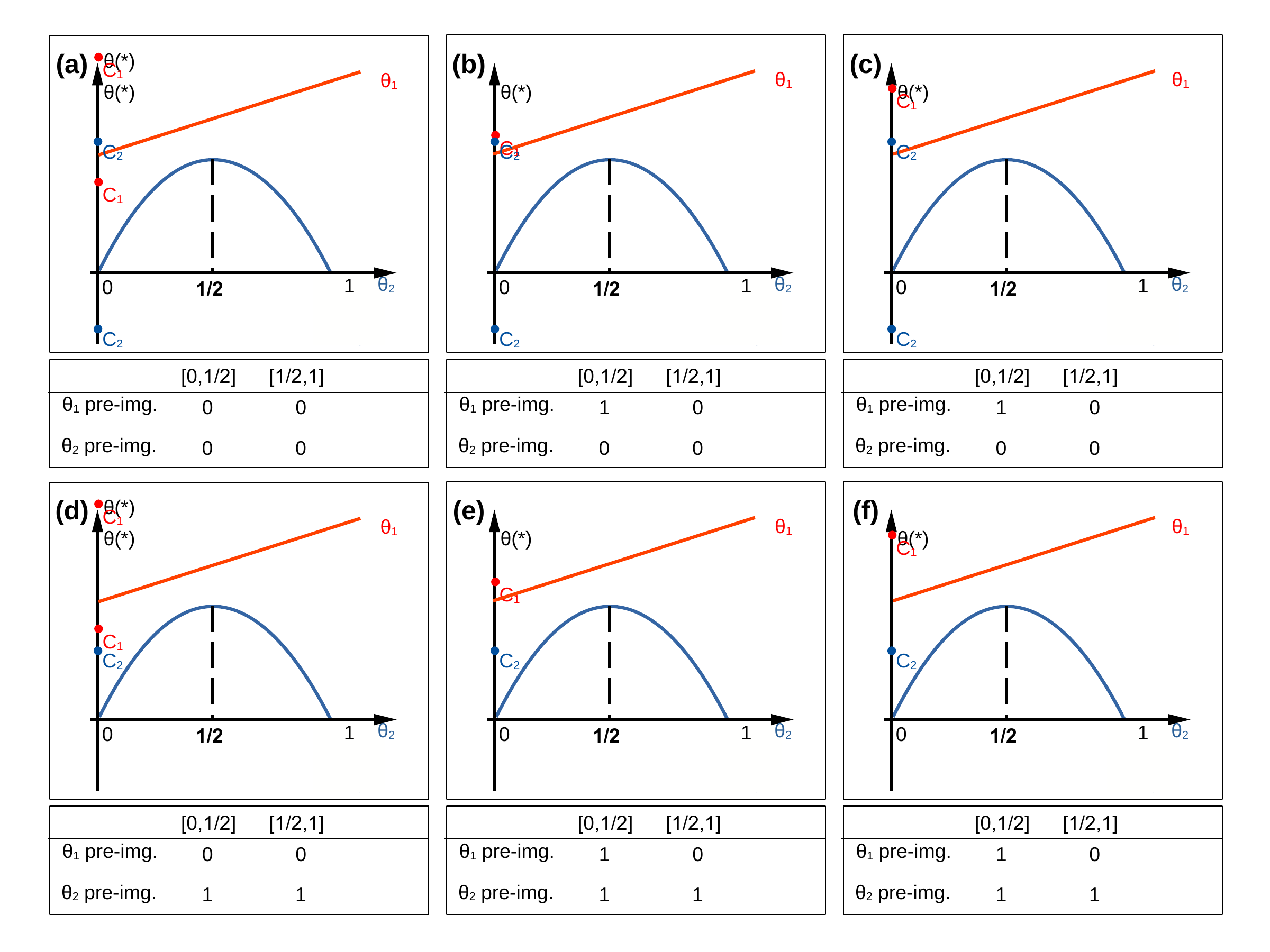}

\caption{\label{fig:Refined-monotonic-intervals} Refined monotonic intervals
$\mathcal{I}_{2}=\{[0,1/2],[1/2,1]\}$ for the $\theta_{1}(x)=x$,
$\theta_{2}(x)=-4x^{2}+4x$ shown. We use vertical black dashed lines
to illustrate the refined monotonic intervals, and count the number
of pre-images for $\theta_{1},\theta_{2}$ over each refined intervals. }
\end{figure}

\begin{example}
\label{exa:Let-us-consider}Let us consider the following three cases
with fixed $\theta_{1}(x)=x+1.2$, $\theta_{2}(x)=-4x^{2}+4x$, with
$\mathcal{I}_{1}=\{[0,1]\}$, $\mathcal{I}_{2}=\{[0,1/2],[1/2,1]\}$,
and the refined monotonic intervals $\mathcal{I}_{1\cap2}=\{[0,1/2],[1/2,1]\}$.
And we consider expressions in \eqref{eq:C1eqs}, \eqref{eq:C2eqs}
along with $\bm{z}_{i,k_{1}}=\theta_{1}\bm{x}_{i}$ and $\bm{z}_{i,k_{2}}=\theta_{2}\bm{x}_{i}$,
as illustrated by Figure \ref{fig:Refined-monotonic-intervals}. %
We can consider the following cases:

Let us consider $I=[0,1/2]$ first, if we fixed $C_{1}\in(-\infty,1.2)\cup(2.2,\infty)$
and choose $C_{2}\in(-\infty,0)\cup(1,\infty)$ as shown in the
(a) in Figure \ref{fig:Refined-monotonic-intervals}. We cannot differentiate
between $\theta_{1}$ and $\theta_{2}$ based solely on the principal
decision ratio when splitting on inputs $\bm{x}\in I$, since both
$\theta_{1}$ and $\theta_{2}$ will have 0 pre-image in $I$, so
the corresponding splits associated with $C_{1},C_{2}$ have the same
chance of being selected on this interval. However, if we choose $C_{1}\in(1.2,1.7)\text{ or }[1.7,2.2),C_{2}\in(-\infty,0)\cup(1,\infty)$
in such a way shown as (b) or (c), then over the $I=[0,1/2]\text{ or }[1/2,1]$ we have a higher
chance of selecting $\theta_{1}$ from analysis of expressions in
\eqref{eq:C1eqs}, \eqref{eq:C2eqs}. If we fixed $C_{1}\in(-\infty,1.2)\cup(2.2,\infty),C_{2}\in[0,1]$
this analysis remains the same and we learn that (d) in Figure \ref{fig:Refined-monotonic-intervals}
can differentiate $\theta_{1},\theta_{2}$.

Now if we choose $C_{1}\in[1.2,2.2],C_{2}\in[0,1]$, then $\theta_{1},\theta_{2}$ will both 
have 1 pre-image in $I$ as shown in (e) and (f) in 
Figure \ref{fig:Refined-monotonic-intervals}. Then we are back to the computation of \eqref{eq:MH_min_2tree}.%
The detailed analysis
will be given in Example \ref{exa:(-calculation)-Suppose}. 
\end{example}

The observation in the above example can be summarized as the following
result linking the principal decision ratios and the refined intervals
of a given univariate feature mapping. 
\begin{prop}
\label{prop:piecewise-monotonic-var-1}Consider one splitting variable
$\bm{x}_{k}$ for a fixed $k\in\{1,2,\cdots,d\}$ and two piece-wise
strictly monotonic transformations $\theta_{1}$ and $\theta_{2}$
with refined monotonic intervals $\mathcal{I}_{1\cap2}$. For any
$C_{1},C_{2}\in\mathbb{R}$ and any $I\in\mathcal{I}_{1\cap2}$, we
have

(i) When $\theta_{1}$ and $\theta_{2}$ have 0 pre-image for $C_{1}$
and $C_{2}$ on $I$, the principal decision ratio is 1 for the transformed
variates $\theta_{1}\bm{x}_{k}$ and $\theta_{2}\bm{x}_{k}$ over
this interval $I\in\mathcal{I}_{1\cap2}$.

(ii) When $\theta_{1}$ and $\theta_{2}$ have different numbers of
pre-images for $C_{1}$ and $C_{2}$ on $I$, the principal decision
ratio is larger for the transform with 1 pre-image over this interval
$I\in\mathcal{I}_{1\cap2}$. 

(iii) When $\theta_{1}$ and $\theta_{2}$ have 1 pre-image for $C_{1}$
and $C_{2}$ on $I$, the principal decision ratio is larger for the
transform $\theta_{i}$ that solves the following problem:
\begin{align*}
\min_{j=1,2} & \sum_{i=1}^{n}(y_{i}-\mu_{L}^{j})^{2}\bm{1}(\bm{x}_{i,k_{j}}\leq\theta_{j}^{-1}C_{j})+\sum_{i=1}^{n}(y_{i}-\mu_{R}^{j})^{2}\bm{1}(\bm{x}_{i,k_{j}}>\theta_{j}^{-1}C_{j})
\end{align*}
When $k_{1}=k_{2}=k$ and $C_{1}=C_{2}=C$, this reduces to \eqref{eq:MH_min_2tree}.
\end{prop}

\begin{proof}
See Appendix \ref{sec:Proof-prop13}. 
\end{proof}
Using the characterization in Proposition~\ref{prop:piecewise-monotonic-var-1},
the probability of $\theta_{1}$ having a larger principal decision
ratio when compared to $\theta_{2}$ can be calculated as follows.
Let $x$ be the $k$-th coordinate of $\bm{x}$, i.e., $\bm{x}_{k}$.
In practice we usually use $C_{1}=C_{2}=C$, and it is not hard to
exclude the case (iii) in the above Proposition \ref{prop:piecewise-monotonic-var-1}:
Example \ref{cor:monotonic_f_0} tells us that $\theta_{1}+C_{0}$
and $\theta_{1}$ for $C_{0}\in\mathbb{R}$ will not change principal
decision ratio $\tau$. Assuming that $\sup\theta_{2}$ and $\inf\theta_{1}$
are finite, we can pick $C_{0}=\sup\theta_{2}-\inf\theta_{1}+1$ to
ensure that any $C_{1}=C_{2}=C$ will not lead to the case (iii).
\begin{prop}
\label{prop:Under-the-assumption}Under the assumption of Proposition
\ref{prop:piecewise-monotonic-var-1} and there is not $I$ where
both $\theta_{1},\theta_{2}$ both have 1 pre-images, assume additionally
that the input location $\bm{x}$ is distributed as $\mathbb{P}$
and denote the $N_{1}+N_{2}$ refined monotonic intervals as 
\begin{itemize}
\item $I_{1}^{1},\cdots,I_{1}^{N_{1}}$ where $\theta_{1}$ has 1 pre-image
and $\theta_{2}$ has 0 pre-image of $C$ ($\theta_{1}$ has higher
chance of being selected); 
\item $I_{2}^{1},\cdots,I_{2}^{N_{2}}$ where $\theta_{2}$ has 1 pre-image
and $\theta_{1}$ has 0 pre-image of $C$ ($\theta_{2}$ has higher
chance of being selected). 
\end{itemize}
Then, the probability that the principal decision ratio prefers $\theta_{1}$
over $\theta_{2}$ can be computed as 
\begin{equation}
p_{1>2}=\sum_{i=1}^{N_{1}}\int_{I_{1}^{i}}1d\mathbb{P}-\sum_{j=1}^{N_{2}}\int_{I_{2}^{j}}1d\mathbb{P}=\mathbb{P}\left(\cup_{i=1}^{N_{1}}I_{1}^{i}\right)-\mathbb{P}\left(\cup_{j=1}^{N_{2}}I_{2}^{j}\right).\label{eq:p1>2}
\end{equation}
\end{prop}

Note that these intervals $I_{1}^{1},\cdots,I_{1}^{N_{1}}$ and $I_{2}^{1},\cdots,I_{2}^{N_{2}}$
are dependent on the fixed transformation pair $(\theta_{1},\theta_{2})$
and it is possible to optimize over $C_{1}$ and $C_{2}$ to maximize
(or minimize) the quantity $p_{1>2}$ in \eqref{eq:p1>2}. 
\begin{example}
\label{exa:(-calculation)-Suppose}($p_{1>2}$ calculation) Suppose
that $\mathbb{P}$ is the uniform measure on $[0,1]$, and consider
$\theta_{1}(x)=x+1.2$, $\theta_{2}(x)=-4x^{2}+4x$ with $C_{1}=C_{2}=C$
as in Example \ref{exa:Let-us-consider}, and the refined monotonic
intervals $\mathcal{I}_{2}=\{[0,1/2],[1/2,1]\}$. Applying \eqref{eq:p1>2}
leads to

\begin{adjustbox}{center}%
\begin{tabular}{cccccc}
\toprule 
$C\in(-\infty,0)$  & $C\in(0,1)$ & $C\in[1,1.2)$ & $C\in[1.2,1.7)$  & $C\in[1.7,2.2)$  & $C\in(2.2,\infty)$ \tabularnewline
\midrule
\midrule 
case (i)  & $I_{2}^{1}=[0,1]$  & case (i)  & $I_{1}^{1}=[0,1/2]$  & $I_{1}^{2}=[1/2,1]$  & case (i) \tabularnewline
\bottomrule
\end{tabular}\end{adjustbox} \medskip{}
 \vspace{2pt}

For other $C$'s, if we assume $\mathbb{P}$
is uniform, probabilities $p_{1>2}$ can be filled in as: \hspace{2pt}

\begin{adjustbox}{center}%
\begin{tabular}{cccccc}
\toprule 
$C\in(-\infty,0)$  & $C\in(0,1)$ & $C\in[1,1.2)$ & $C\in[1.2,1.7)$  & $C\in[1.7,2.2]$  & $C\in(2.2,\infty)$ \tabularnewline
\midrule
\midrule 
0. & 1. & 0. & 0.5  & 0.5  & 0.\tabularnewline
\bottomrule
\end{tabular}\end{adjustbox} \vspace{2pt}

With this table, one can observe that the value of $p_{1>2}$ can
be maximized by choosing $C\in[1.2,2.2]$; and it can be minimized
by choosing $C\in(0,1)$. 
\end{example}

\section{\label{sec:Global-rankings-with}Global Rankings with Regressions}

The ranking perspective established in preceding sections is focused on local splits. We now turn to extending this perspective to study the global performance of tree-based methods, including both a single tree consisting of multiple splits and tree ensembles, in terms of ranking. In this section,   
we always consider the full dataset of size $N$ instead of node-specific
sample size $n$. On this full dataset, we will show that tree-based
methods such as CART and BART, when trained in a supervised regression
context, yield good ranking performance.

Throughout this section, we consider the model \eqref{eq:noisy score-1}
as a ``noisy scoring'' model. In particular, the mean
function $f$, now considered as a scoring function, takes the feature $\bm{x}_{i}$ and assign it a ``score''
$f(\bm{x}_{i})$, and its noisy ``score'' is $y_{i}$. Any scoring function $f$
can induce a permutation $r=\{j_{1},j_{2},\cdots,j_{N}\}$ 
on the full set of $\{1,2,\cdots,N\}$ with
the same cardinality, such that $f(\bm{x}_{j_{1}})\geq f(\bm{x}_{j_{2}})\geq\cdots\geq f(\bm{x}_{j_{N}})$. We consider the following criterion to assess the ranking performance of scoring functions $f$ through $r$: 
\begin{align}
\bm{T}\left(r\right) & =\frac{2}{N(N-1)}\sum_{i=1}^{N-1}\sum_{i'=i+1}^{N}(\mathbb{E}_{y_{j_{i}}\mid\bm{x}_{j_{i}}}y_{j_{i}}-\mathbb{E}_{y_{j_{i'}}\mid\bm{x}_{j_{i'}}}y_{j_{i'}}),\label{eq:T-metric}
\end{align}
which has been studied in the ranking literature such as \cite{cossock2006subset} and can be also generalized to operate on subsets instead of the full dataset. Note that previously we considered rankings with actual responses $y_{i}$'s; in the presence of noise,
the metric in~\eqref{eq:T-metric} consider orderings on the \emph{conditional means} 
instead of orderings on the responses $y_{j}$ in \eqref{eq:y_rank convention}
to avoid the noise effect.

The \emph{Bayes-scoring
function} $f_{B}(\bm{x}_{j})=\mathbb{E}_{y_{j}\mid\bm{x}_{j}}y_{j}$
is defined as the conditional expectation of $y_{j}$ conditioning
on the $j$-th input $\bm{x}_{j}$. Its incuded 
permutation, denoted by $r_B$, maximizes $\bm{T}\left(r\right)$ in~\eqref{eq:T-metric} (see, e.g., \citet{cossock2006subset}). Hence, 
the Bayes rank can be defined by this permutation $r_{B}$ that sorts
the conditional means, and the \eqref{eq:T-metric} measures how any
other permutation deviates from this ``optimal permutation''.
From this discussion,
we can see that the optimal rank-preserving function is not unique
and can be obtained from pre-composite monotonic transformations to $f_{B}$ like $\tau\circ f_{B}$. 

We are now in a position to establish oracle bounds under the ranking metric $\bm{T}\left(r\right)$ for a fully-split single CART tree with \textit{finite samples}, based on the oracle properties from \citet{klusowski2021universal}. Consider the function class $\mathcal{G}$ that collects functions with an additive form
$f(\bm{x})=\sum_{i=1}^{d}f_{i}(\bm{x}_{i})$, where each coordinate
$f_{i}:\mathbb{R}\rightarrow\mathbb{R}$ has
bounded variation (hence $f$ has bounded variation), and the $\mathcal{G}_0$ is a pre-chosen model class that might deviate from $\mathcal{G}$ as set up in \citet{klusowski2021universal}, which allows for possible model misspecification. Here we need this pre-chosen model class $\mathcal{G}_0$ to contain $\mathcal{G}$. We consider a random design where $\mathcal{X}^{(N)}$ is a simple random sample of a distribution $\mathbb{P}_X$. For a function $f$, its $\ell_2$ norm is defined as $\|f\|^{2} = \int f(u)^{2}d\mathbb{P}_X(u)$, and its supremum norm is denoted as $\|f\|_{\infty}$.

\begin{thm}
\label{thm:Oracle}  (Oracle inequality for ranking)  Suppose that the Bayes scoring function $f_{B} \in \mathcal{G}_0 \supset\mathcal{G}$ and that we have a complete binary regression tree $g_{c,K}$ of depth $K\geq1$
constructed by CART.
Then the permutation $r_{c,K}$ induced by $g_{c,K}$ satisfies 

\begin{align}
 & \nonumber \mathbb{P}_{\left(\mathcal{X}^{(N)},\mathcal{Y}^{(N)}\right)}\left(\left.\bm{T}\left(r_{B}\right)-\bm{T}\left(r_{c,K}\right)>\varepsilon_{N}\right.\right)\\
\leq & %
\frac{1024}{\varepsilon_{N}^{4}}\cdot\inf_{g\in\mathcal{G}}\left\{ \left\Vert f_{B}-g\right\Vert ^{2}+\frac{\left\Vert g\right\Vert _{\text{TV}}^{2}}{K+3}+C_{B}\frac{2^{K}\log^{2}N\log(Nd)}{N}\right\} +\frac{32R_{\infty}J_{0}(2R_{2},\mathcal{G}_0)}{\varepsilon_{N}^{2}\sqrt{N}}, \label{eq:rank.bound.CART.two.class} %
\end{align}
for any sequence $\varepsilon_{N}>0$ that tends to 0, where %
$R_{\infty}=2\sup_{f\in\mathcal{G}_0}\|f\|_{\infty},R_{2}=2\sup_{f\in\mathcal{G}_0}\|f\|,$ $J_{0}(2R_{2},\mathcal{G}_0)$ is the entropy as (2.2) in \citet{van2014uniform} for the function class $\mathcal{G}_0$, and the constant  
$C_{B}$ depends on the supremum norm of 
$f_{B}$ and the class $\mathcal{G}$. 

When $\mathcal{G}_0=\mathcal{G}$, the upper bound in~\eqref{eq:rank.bound.CART.two.class} can be simplified to 
\begin{equation}
\label{eq:oracle.special.case}
\frac{1024}{\varepsilon_{N}^{4}}\cdot\left\{ \frac{\left\Vert f_B\right\Vert _{\text{TV}}^{2}}{K+3}+C_{B}\frac{2^{K}\log^{2}N\log(Nd)}{N}\right\}+\frac{32R_{\infty}J_{0}(2R_{2},\mathcal{G})}{\varepsilon_{N}^{2}\sqrt{N}},%
\end{equation}
and the constant $C_{B}$ only depends on the total variation of $f_B$.
\end{thm}

\begin{proof}
See Appendix \ref{sec:Proof-of-Theorem-Oracle}. 
\end{proof}

The right-hand side of \eqref{eq:rank.bound.CART.two.class} decomposes the first term inside the infimum into three components:
the possible approximation error $\left\Vert f_{B}-g\right\Vert ^{2}$
induced by finite depth $K$; the total variation caused by the tree
approximant $g$; the decaying term showing that $d$ can grow exponentially
without losing the consistency of a CART fit. The second term bounds the difference between the empirical norm and the $\ell_2$ norm. 

Theorem~\ref{thm:Oracle} holds for any sample size. 
This theorem yields asymtpotic rates by letting
$N$ diverge; we can see that the obtained rate $\epsilon_{N}$ is
not faster than $N^{-1/4}$, which is much slower than that obtained
in the following Theorem \ref{thm:Asymptotics}. %
Substituting $\varepsilon_{N}\asymp O\left(N^{-1/4-\delta}\right)$
yields the following consistency result for ranking. 
\begin{cor}
\label{cor:Consistency}
Under the same conditions as in Theorem \ref{thm:Oracle}, if we assume
that the depth $K=K_{N}$ of the tree grows
with the sample size $N$ in such a way that 
\begin{align*}
    \left\Vert f_{B}\right\Vert _{\text{TV}}^{2}\asymp o \left(\varepsilon_{N}^{4}\sqrt{K}\right),\frac{2^{K}\log^{2}N\log(Nd)}{N}\asymp o \left(\varepsilon_{N}^{4}\right),
\end{align*}
where $\varepsilon_{N}^{4}N\rightarrow\infty$
(i.e., $\varepsilon_{N}\asymp O \left(N^{-1/4-\delta}\right)$
for $\delta>0$) as $N\rightarrow\infty$, then there holds  
\begin{align*}
\lim_{N\rightarrow\infty}\mathbb{E}_{\left(\mathcal{X}^{(N)},\mathcal{Y}^{(N)}\right)}\left(\left| \bm{T}(r_B)-\bm{T}(r_{c,K})\right| \right) \rightarrow0.
\end{align*}
\end{cor}
\begin{rem}
The assumptions $\left\Vert g\right\Vert _{\text{TV}}^{2}\asymp o \left(\varepsilon_{N}^{2}\sqrt{K}\right),\frac{2^{K}\log^{2}N\log(Nd)}{N}\asymp o \left(\varepsilon_{N}^{2}\right)$
are parallel to the assumptions imposed by Corollary 4.4 in \citet{klusowski2021universal}.
\end{rem}

 The next result shows that,
under conditions, BART has a posterior concentration close to $f_{B}$ if
the $y_{i}$'s are generated by \eqref{eq:noisy score-1}. We consider a fixed and regular design as described in Definition 7.1 of \citet{rovckova2019theory} or Definition 3.3 in \citet{rovckova2020posterior}. In the context of \citet{rovckova2020posterior}, a \emph{regular design} refers to a fixed dataset where the diameters of the cells in a k-d tree partition are controlled and relatively uniform. Specifically, the maximal diameter in the partition components should not be significantly larger than a typical diameter. An example of a regular dataset would be a fixed design on a regular grid, where the points are evenly spaced and no cells have an excessive spread of points compared to others. In contrast, a dataset with highly skewed or isolated points in certain directions might not meet the regularity condition.

We use
the probability measure corresponding to
responses generated using the Bayes scoring function $f_{B}$
in model \eqref{eq:noisy score-1}, which implies that the response should be considered evaluated at these fixed inputs with random noise. 
\begin{thm}
\label{thm:Asymptotics} (Fixed design) Assume that the Bayes scoring function $f_{B}$
is $\nu$-Holder continuous for $\nu\in(0,1]$, with the norm $\|f_{B}\|_{\infty}\apprle\log^{1/2}N$ and 
a regular design over the set $\mathcal{X}^{(N)}=\{\bm{x}_{1},\cdots,\bm{x}_{N}\}\subset\mathbb{R}^{d}$
where $d\lesssim\log^{1/2}N$.  Let the function class $\mathcal{F}$ be defined as a set of additive
simple functions as described in (3) of \citet{rovckova2019theory}. Consider the BART prior with a fixed number of
trees and node splitting probability $p_{split}(\eta)=\alpha^{\text{depth}(\eta)}$
for a node $\eta$ and $\alpha\in\left[\frac{1}{N},\frac{1}{2}\right)$. Then,  the following contraction for BART posterior $\prod$ holds
for the resulting posterior distribution and the BART induced ranking functions $r_{f}$
from \eqref{eq:noisy score-1}:
\begin{align*}
\prod\left(\left.f\in\mathcal{F}:\bm{T}\left(r_{B}\right)-\bm{T}\left(c_{f}\right)>M_{N}\varepsilon_{N}\right|y_{1},\cdots,y_{N}\right) & \rightarrow0
\end{align*}
in probability measure of the $y_1,\cdots,y_N$, where $\varepsilon_{N}=N^{-\alpha/(2\alpha+d)}\log^{1/2}N$, and for any sequence $M_{N}\rightarrow\infty$,
as the sample size $N \rightarrow \infty$ and the dimensionality $d\rightarrow\infty$. 
\end{thm}

\begin{proof}
See Appendix \ref{sec:Proof-of-Asymptotics}. 
\end{proof}
\begin{rem}
The Bayes scoring function
$f_{B}(\bm{x}_{j})$ always exists (see Theorem 1 in \citet{cossock2006subset}). Theorem~\ref{thm:Asymptotics} indicates that as we fit BART with an increasing number of samples $\mathcal{X}$ satisfying
a regular design, the resulting posterior will concentrate around
the Bayes scoring function. Unlike the finite-sample result for a single binary tree established in Theorem~\ref{thm:Oracle}, Theorem~\ref{thm:Asymptotics} provides an asymptotic result concerning the posterior distribution of BART.
\end{rem}

Building on our previous discussions from a ranking perspective,
we can summarize the findings as follows: Locally, at each split, the partitions are most likely divided
into rank-consistent groups, but within each partition,
no ranking is available since the scoring function remains constant
within each partition. Globally, in the asymptotic behavior
of BART, the posterior tends to concentrate near the Bayes scoring functions
that minimize the $L_{2}$ (hence $\bm{T}$ metric) error. Meanwhile, CART also achieves consistent ranking performance, with finite-sample bounds available.

\section{\label{subsec:Concordant--Statistics}Concordant Divergence}
We now shift to leveraging the ranking perspective to study symbolic feature selection as illustrated in Example~\ref{exa:(Symbolic-feature-mappings)}. Although tree-based methods have demonstrated strong finite-sample performance in distinguishing between symbolic features that are transformations of one another, it remains elusive since the existing theory, viewed through the lens of nonparametric variable selection, is unaffected by such transformations. In this section, we extend our previous analysis to uncover additional characteristics of tree-based methods. Building on these insights, we introduce a \emph{concordant divergence} statistic, $\mathcal{T}_0$, which can evaluate feature mappings.

The local split analysis in previous sections has provided some insight into tree-based methods when transformations are involved. Section \ref{subsec:Response-rankings-with} shows that the oracle partition, relevant only to the ranks of responses $y$'s, may not be attainable with a single split
along the input domain, unless there is monotonicity along one coordinate. Section \ref{subsec:Piece-wise-monotonic-transforms} compares two piecewise transformations in a relative sense. Next, we first present another local-level observation before extending the ranking perspective to study a general mapping $g$ in an absolute sense, moving beyond local splits. 
\begin{lem}
\label{lem:(Swaps-to-optimize}(Magnitude of swaps) Under the same
assumptions as in Lemma \ref{lem:LemmaA}, we suppose that the only reversed
pairs are $(y_{\alpha},y_{\gamma})$ and $(y_{\beta},y_{\gamma})$
where $y_{\alpha}>y_{\gamma},y_{\beta}>y_{\gamma}$ and both $y_{\alpha},y_{\beta}\in P_{1}$
but $y_{\gamma}\in P_{2}$. If $y_{\alpha}>y_{\beta}>y_{\gamma}$,
then the swap for the pair $(y_{\alpha},y_{\gamma})$ reduces the loss
\eqref{eq:loss.y.ranking} more than the swap for the pair $(y_{\beta},y_{\gamma})$
. 
\end{lem}

\begin{proof}
See Appendix \ref{sec:Proof-of-LemmaA1}. 
\end{proof}
Lemma \ref{lem:LemmaA} highlights the effect of
the magnitude of the responses. In particular, we might prioritize the swapping of a reverse pair
with a larger ``size'', i.e., the difference between the responses
$y$ in the reversed pair. It is possible that there exist two reversed
pairs with the same ``sizes''. However, with our assumption that
both inputs and responses are continuous, it is with zero probability
that we have two reverse pairs such that their magnitudes are identical.

Based on the discussion of tree-like models for the univariate
case (Lemmas \ref{lem:LemmaA} and \ref{lem:(Swaps-to-optimize}),
we can summarize the principle behind feature selection as follows: it
selects the feature mapping $g$ that takes $\bm{x}$ into transformed
images $\bm{z}$ that have the most similar rankings as the response
$y$. The discrepancies between these rankings can be described by the ``gaps'' between
ordered statistics of $y$'s. Motivated by Lemma \ref{lem:(Swaps-to-optimize},
we develop $\mathcal{T}_{0}$ to evaluate arbitrary feature
mapping $\theta$:
\begin{equation}
\begin{split}
\mathcal{T}_{0}(\theta) =\sum_{\pi}\frac{2|y_{\pi(1)}-y_{\pi(2)}|}{n(n-1)}\cdot \text{\{} \bm{1}(\theta(\bm{x}_{\pi(1)}) & \geq \theta(\bm{x}_{\pi(2)}))\cdot\bm{1}(y_{\pi(1)}<y_{\pi(2)})   \\
&  +\bm{1}(\theta(\bm{x}_{\pi(1)})<\theta(\bm{x}_{\pi(2)}))\cdot\bm{1}(y_{\pi(1)}\geq y_{\pi(2)}) \text{\}},\label{eq:T0_statistics}\\
\end{split}
\end{equation}
where the summation takes over all permutations of length 2 as $(\pi(1),\pi(2))$
with $\pi(1),\pi(2)\in\{1,\cdots,n\}$. Guided by the ranking behavior induced by tree methods, $\mathcal{T}{0}(\theta)$ evaluates a symbolic feature $\theta(\bm{x})$ by measuring to what extent it can recover the order of $y$’s. In particular, if $\theta(\bm{x}{(i)}) < \theta(\bm{x}{(j)})$, we have a zero summand; otherwise, we will have a non-negative summand $\left|y{(i)} - y_{(j)}\right|$. A larger $\mathcal{T}_{0}(\theta)$ indicates more discrepancy between the rankings of $\theta(\bm{x})$ and $y$, and we accumulate all these discrepancies. Note that despite the intimate connection with the ranking performance of tree methods, to use this divergence, we do not need to consider a tree model, but simply a finite number of samples.

\begin{rem}
The $\mathcal{T}_{0}$ is
similar to Kendall's tau~\citep{hollander2013nonparametric}
but with the additional non-negative
multiplier,  $\left|y_{\pi(1)}-y_{\pi(2)}\right|$, representing ``the magnitude of swaps''. %
It involves both the ranks and the actual values of the responses $y$. This $\mathcal{T}_{0}$ is also not the same as linear correlation coefficients. \citet{daniels_relation_1944}
(in Section 5) stated that the linear correlation coefficients $\rho$
 satisfy $\rho_{\theta x,y}=\rho_{\theta x,x}\rho_{x,y}$ for
any transformation $\theta$, which means that the correlation $\rho_{\theta x,y}$
cannot increase beyond $\rho_{x,y}$ since $\rho_{\theta x,x}\leq1$.
However, the behavior of $\mathcal{T}_{0}$ is not constrained in the same way when transformations
are introduced. 
\end{rem}

For inactive variables, we want to exclude both the variable itself and all its transformations.
From the following definition, it is clear that any transformation
of inactive variables will also remain inactive. 
\begin{defn}
(Inactive variable) A feature $\bm{X}_{k}\in\mathbb{R}$
of a continuous input variable $\bm{X},k\in\{1,\cdots,d\}$ is called
inactive (for function $f$ as in \eqref{eq:noisy score-1}), if the distribution of $y$ is independent of the distribution
of $f(\bm{X}_{k})$. 
\end{defn}

Considering a random design where each row of $\bm{X}$ is drawn independently from a distribution, we have the following properties for $\mathcal{T}_0$ when evaluating transformations of $\bm{X}_k$: 
\begin{prop}
\label{prop:-Conditioned-on} %

(i) if $\bm{X}_{k}$ is an inactive variable, then $\mathbb{E}_{\bm{X}_k, \bm{y}}\mathcal{T}_{0}(f)\not\rightarrow0$
as $n\rightarrow\infty$.

\noindent(ii) if there exists a transformation $\theta=g$ such that $g(\bm{x}_{1})\geq g(\bm{x}_{2})\Leftrightarrow$$y_{1}\geq y_{2}$,
then $\mathbb{E}_{\bm{X},\bm{y}}\mathcal{T}_{0}(g)=0$. 
\end{prop}

\begin{proof}
See Appendix \ref{sec:Proof-of-Proposition-T0}. 
\end{proof}
These two results establish the fact that $\mathcal{T}_0$ will
never prefer a mapping consisting of an inactive variable (Proposition
\ref{prop:-Conditioned-on}, i), unless that mapping is a ``fake
interpolator'' for the given finite sample (Proposition \ref{prop:-Conditioned-on}, ii and Example
\ref{exa:(Interpolator)-Suppose-that}). In an extreme case where
all variables are inactive, that is, none of the coordinates
in $\bm{X}$ determine the value of $y$, part (i) of the proposition
implies that $\mathcal{T}_{0}$ will not be zero. This means that if
we take $g=\text{id}$, then such an identity mapping is not an interpolator
(see Example \ref{exa:(Interpolator)-Suppose-that}). However, this
does not rule out the possible existence of an interpolator $g\neq\text{id}$
such that $g$ ``reorganizes'' $\bm{X}$ and $y$ in a concordant
way, even if $\bm{X}$ are completely inactive.

As we do not actually have to use a tree structure when computing
$\mathcal{T}_{0}$, this gives us convenience in practice. 
In symbolic feature selection, we will calculate the correlation $\mathcal{T}(y(\bm{z}))$
between the response $y_{i}$ and each coordinate of $\bm{z}_{i}$' to find useful features (i.e., correlation between the $\bm{y}$
vector and the $q$ columns of $\bm{z}$ matrix). 
\begin{example}
\label{exa:(Power-of-)}(Power of $\mathcal{T}_{0}$) In Figure \ref{tab:-in-T0},
we present 5 different feature mappings $\theta_{1},\cdots,\theta_{5}$
and compute their pairwise correlations and $\mathcal{T}_{0}$
and compare its performance to other correlations. For ease of
comparison, we use log-scale for $\mathcal{T}_{0}$  and [0,1] scale for the other correlations.  
We will expect the divergence to be close to 0, indicating dependence between the sample $\bm{x}$ and $\theta_i(\bm{x})$  to various extents (i.e., $\mathcal{T}_{0}=0$ as in Proposition \ref{prop:-Conditioned-on} (ii)). 
Since $\theta_{1}=\bm{x}$ coincide with $\theta_{2},\cdots,\theta_{5}$ to different extents, we also want the correlation reflect the degree of dependence. 
\end{example}

The \citet{chatterjee2021new}'s correlation coefficient $\xi_{n}(\bm{X},\bm{y})$
and the $\mathcal{T}_{0}$ are both asymmetric and measure both capture
non-linear dependencies between pairs of random variables, particularly
non-linear dependencies. The $\xi_{n}(\bm{X},\bm{y})$ rearranges
data pairs based on sorted values and computes rank-based statistics,
making it sensitive to changes in the data's distribution and structure.
On the other hand, $\mathcal{T}_{0}$ is a permutation-based measure
that evaluates the sum of contributions from all possible pairs of
data points, considering differences in values and their rankings.
This exhaustive approach in computing divergence is robust against outliers and provides a
detailed understanding of pairwise dependencies. However, it is computationally
intensive due to the reliance on permutations, especially for moderate to large datasets.

From Figure \ref{tab:-in-T0}, we can observe that classic correlations
like Pearson, Spearman, and Kendall cannot detect the functional dependence
between $\bm{x}$ and the other $\theta_{i}$'s
regardless of the signal-to-noise ratio, which is proportional to $1/\sigma^{2}$.
However, the \citet{chatterjee2021new} correlation and $\mathcal{T}_{0}$
are capable of capturing this dependence when the signal-to-noise
ratio is high and the sample size is sufficiently large ($N=50$). 
In addition, we may also observe that compared to Chatterjee
correlation, $\mathcal{T}_{0}$ will not falsely detect functional
dependence when the signal-to-noise ratio is low, even with only $N=50$. Chatterjee
 correlation seems to stumble when the sample size is limited. $\theta_{1}$ and the rest $\theta_2,\cdots,\theta_5$ have different degree of overlapping, which is reflected by the magnitude of $\mathcal{T}_0$ (when noise variance is small), but not by the other correlations.  %

\begin{figure}[h!]
\centering
\begin{tabular}{ccccc}
\toprule 
$\theta_{1}(x)$  & $\theta_{2}(x)$  & $\theta_{3}(x)$  & $\theta_{4}(x)$  & $\theta_{5}(x)$\tabularnewline
\midrule 
$x$  & $\begin{cases}
+x & x>0\\
x & x\leq0
\end{cases}$  & $-\theta_{1}(x)$  & $\begin{cases}
x+1 & x\in[-1.0,-0.5)\\
-x & x\in[-0.5,0.0)\\
x & x\in[0.0,0.5)\\
-x+1 & x\in[0.5,1.0]
\end{cases}$  & $-\theta_{4}(x)$\tabularnewline
\bottomrule
\end{tabular}
\includegraphics[width=0.95\textwidth]{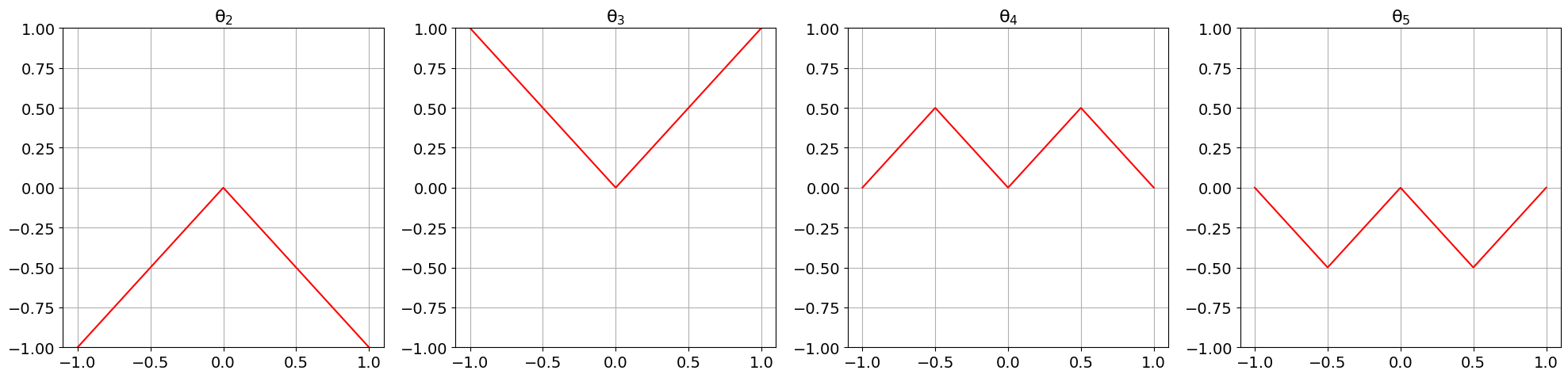}
\vspace{3em}
\begin{tabular}{c}
\toprule 
Sample size N=50\tabularnewline
\midrule 
\includegraphics[width=0.99\textwidth]{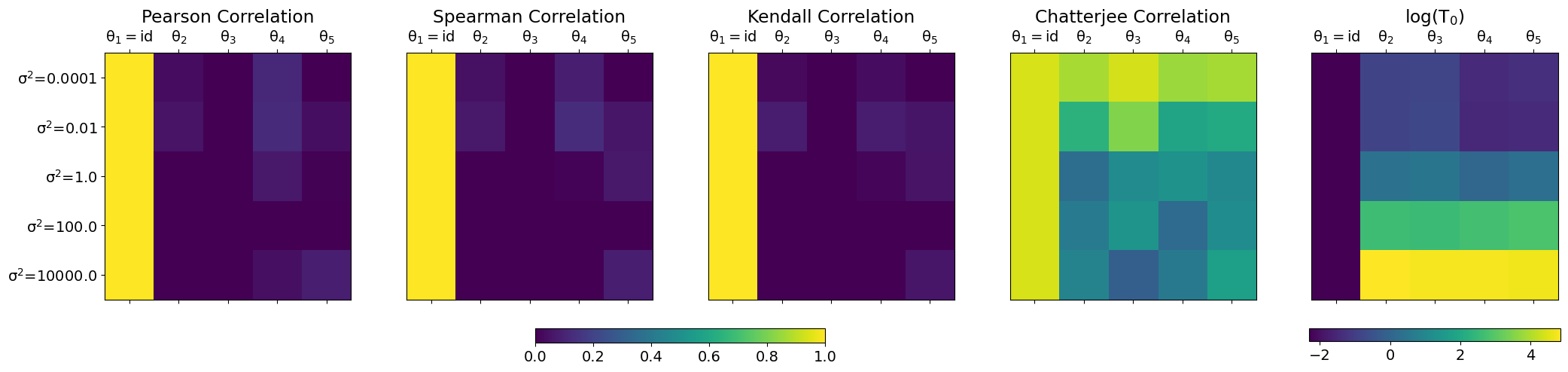}\tabularnewline
\midrule 
Sample size N=500\tabularnewline
\midrule 
\includegraphics[width=0.99\textwidth]{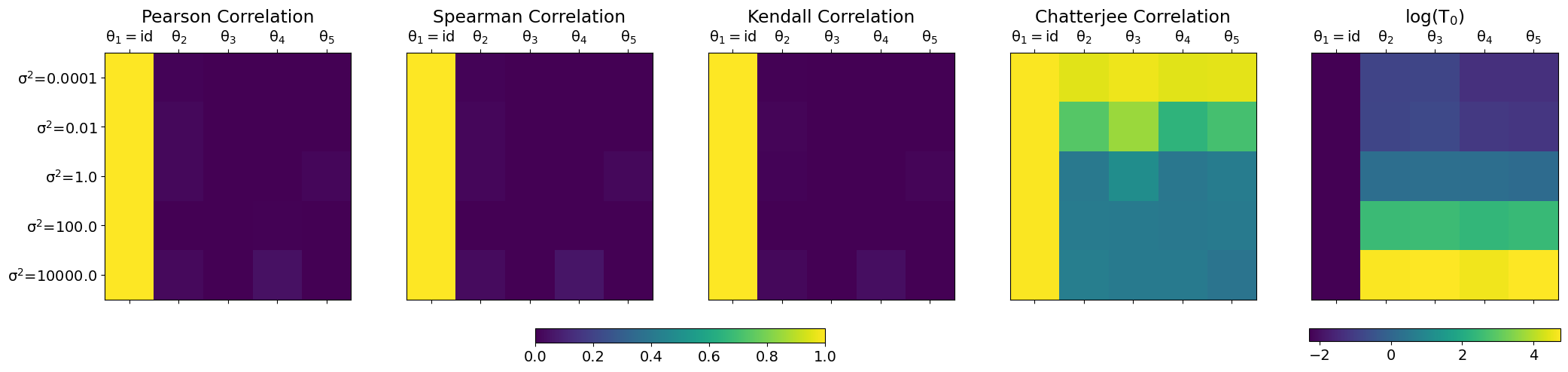}\tabularnewline
\bottomrule
\end{tabular}
\caption{\label{tab:-in-T0}Correlation between $\bm{x}$ and $\bm{y} = \theta_i(\bm{x})$ for $i = 1, \ldots, 5$. The expression and figure for each $\theta_i$ are reported in the top two rows  in the table. Left to Right (in the 3rd and 4th rows): Chatterjee correlation \citep{chatterjee2021new},
absolute Pearson correlation, absolute Spearman correlation and absolute
Kendall correlation, $\log(\mathcal{T}_{0})$. The $\mathcal{T}_{0}$
is shown on a log-scale for better comparison. %
We generate an equally spaced $\bm{x}$ on
$[-1,1]$ with sample size $N=50$ (3rd row) and $N=500$ (4th row). Gaussian noises with variance $\sigma^2$ are added to $\theta_i(\bm{x})$. }
\end{figure}

This example elucidates the behavior of various selection criteria
when applied to four distinct features generated from the same input.
In contrast to Pearson, Spearman and Kendall correlation, which measure
linear and ordinal association respectively, the $\mathcal{T}_{0}$
statistic demonstrates an effective approach. It does not erroneously
filter out the correct function when compared to the correlations
between the true function and the different features $\theta_{i}\bm{x}$.
As shown in the experimental results in Figure \ref{tab:-in-T0},
only $\mathcal{T}_{0}$ can detect the functional dependence and being
sensitive to signal-to-noise ratio; and this makes $\mathcal{T}_{0}$
a suitable correlation of detecting even different sampling plans.
This exemplifies $\mathcal{T}_{0}$'s utility in feature selection,
where the goal is to maintain the true influential features. %

\begin{table}[h!]
\centering

\begin{tabular}{ccccc}
\toprule 
 & $\theta_{1}(x)$  & $\theta_{2}(x)$  & $\theta_{3}(x)$  & $\theta_{4}(x)$\tabularnewline
\midrule
\midrule 
 & $x$  & $2\sin(x)$  & $2\sin(7x)$  & $\sin(x)$ \tabularnewline
\midrule 
BART Global Max  & 18.0  & 10.0  & 10.0  & \textbf{0.0 }\tabularnewline
\midrule 
BART Local  & 29.0  & 29.0  & 25.0  & \textbf{0.0 }\tabularnewline
\midrule 
Pearson  & 0.0  & 100.0  & 0.0  & \textbf{0.0 }\tabularnewline
\midrule 
Kendall  & 100.0  & 0.0  & 0.0  & \textbf{0.0 }\tabularnewline
\midrule 
$\mathcal{T}_{0}$  & 100.0  & 100.0  & 100.0  & \textbf{0.0 }\tabularnewline
\midrule
\midrule 
 & $x$  & $\sin(4x+0.2)$  & $\sin(4x+0.1)$  & $\sin(4x)$ \tabularnewline
\midrule 
BART Global Max  & 0.0  & 0.0  & 5.0  & \textbf{64.0 } \tabularnewline
\midrule 
BART Local  & 0.0  & 0.0  & 9.0  & \textbf{83.0 } \tabularnewline
\midrule 
Pearson  & 0.0  & 0.0  & 0.0  & \textbf{100.0 } \tabularnewline
\midrule 
Kendall  & 0.0  & 0.0  & 0.0  & \textbf{100.0 } \tabularnewline
\midrule 
$\mathcal{T}_{0}$  & 0.0  & 0.0  & 0.0  & \textbf{100.0 } \tabularnewline
\midrule
\midrule 
 & $x$  & $\cos(x)$  & $\sin(2x)$  & $\sin(x)$ \tabularnewline
\midrule 
BART Global Max  & 4.0  & 5.0  & 6.0  & \textbf{5.0 } \tabularnewline
\midrule 
BART Local  & 7.0  & 11.0  & 13.0  & \textbf{13.0 } \tabularnewline
\midrule 
Pearson  & 0.0  & 0.0  & 0.0  & \textbf{100.0} \tabularnewline
\midrule 
Kendall  & 100.0  & 0.0  & 0.0  & \textbf{0.0 } \tabularnewline
\midrule 
$\mathcal{T}_{0}$  & 100.0  & 0.0  & 100.0  & \textbf{100.0 } \tabularnewline
\midrule
\midrule 
 & $x$  & $\sin(4x)$  & $\sin(6x)$  & $\sin(5x)$ \tabularnewline
\midrule 
BART Global Max  & 0.0  & 16.0  & 0.0  & \textbf{61.0 } \tabularnewline
\midrule 
BART Local  & 1.0  & 29.0  & 0.0  & \textbf{68.0 } \tabularnewline
\midrule 
Pearson  & 0.0  & 0.0  & 0.0  & \textbf{100.0 } \tabularnewline
\midrule 
Kendall  & 0.0  & 0.0  & 0.0  & \textbf{100.0 } \tabularnewline
\midrule 
$\mathcal{T}_{0}$  & 0.0  & 0.0  & 0.0  & \textbf{100.0 } \tabularnewline
\bottomrule
\end{tabular}

\caption{\label{tab:-in-18 revisit} The inclusion percentage, as an empirical approximation
to the inclusion probability, by methods BART ($\mathtt{bartMachine}$
R package ($m=50$)), $\mathcal{T}_{0}$, Pearson's correlation  and Kendall's  tau
between $x,\theta x$ are provided for comparison. The true signal
$\theta_{4}$ is highlighted in bold, and all experiments are done
with $\bm{x}\sim N(0,1)$ with sample size 500. }
\end{table}

\section{\label{sec:Experiments}Experiments}

In this section, we conduct simulations to assess the performance
of the proposed $\mT_{0}$ divergence for selecting variables and
symbolic expressions, and lend support to the theoretical results in preceding
sections. We expect that the concordant divergence statistics, which is motivated
as derived from the discussion that ``tree-based splits are attempting
to match the rankings of $\bm{z}$ and $y$'', will behave similarly
to the tree-based methods, on these symbolic regression tasks.

\subsection{\label{3-var-signal}AUC for feature selection }

\begin{figure}[t]
\centering

\begin{tabular}{ccc}
Noise variance 0  & Noise variance 0.01  & Noise variance 0.1\tabularnewline
\includegraphics[width=0.3\textwidth]{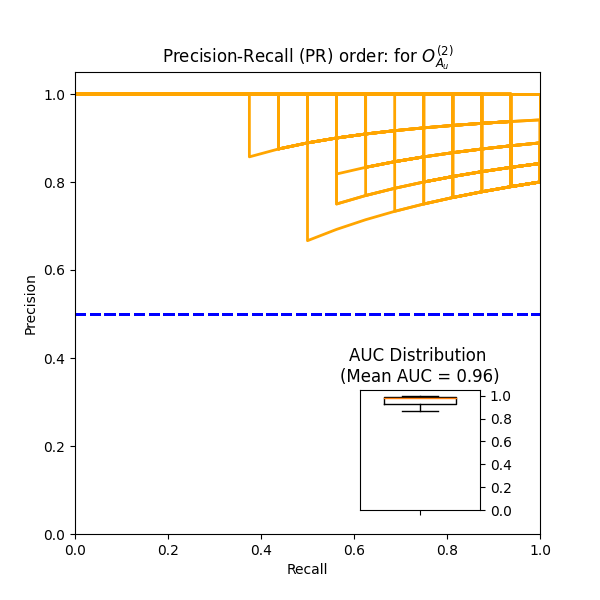}  & \includegraphics[width=0.3\textwidth]{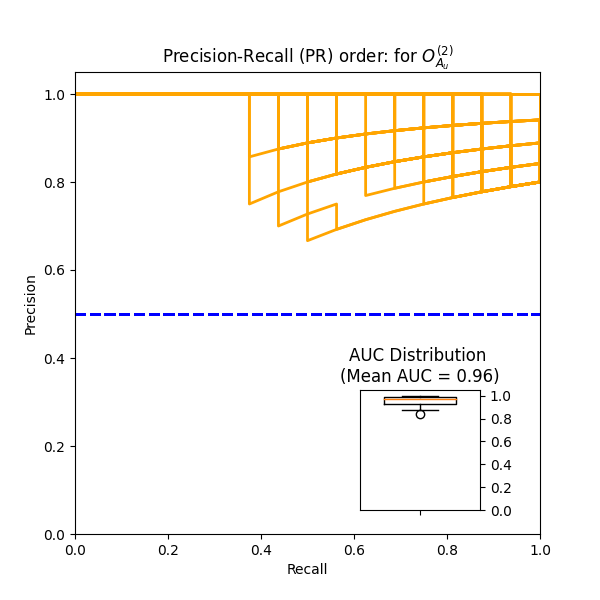}  & \includegraphics[width=0.3\textwidth]{figs/noise_var_0\lyxdot 01_PR_AUC_BU_raw}\tabularnewline
\includegraphics[width=0.3\textwidth]{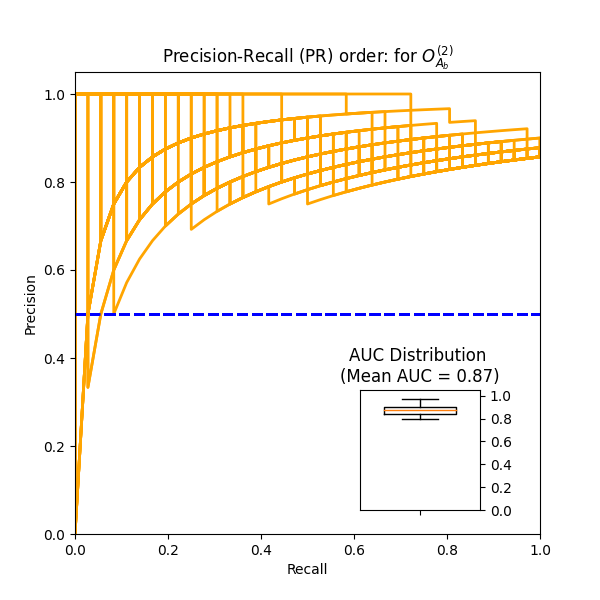}  & \includegraphics[width=0.3\textwidth]{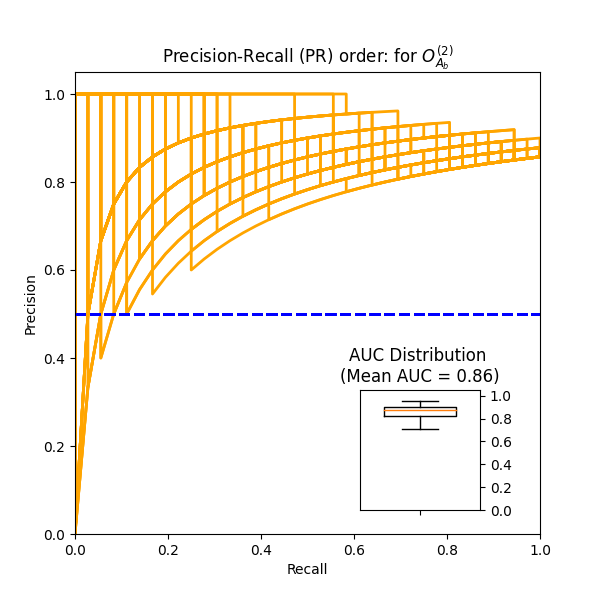}  & \includegraphics[width=0.3\textwidth]{figs/noise_var_0\lyxdot 01_PR_AUC_UB_raw}\tabularnewline
\end{tabular}

\caption{\label{fig:PR_AUC} We illustrate the PR curves from 50 repeats ($n=100$)
of a 2-layer symbolic regression with $\mathcal{O}_{u}=\{id,x^{3}\}$
and $\mathcal{O}_{b}=\{+,\times\}$. The true signal is \eqref{eq:3_var_true_signal}
with no noise. The first row corresponds to the architecture of $\mathcal{O}_{A_{u}}^{(2)}$
and the second row corresponds to the architecture of $\mathcal{O}_{A_{b}}^{(2)}$.
We provide the boxplot to show the AUC values amongst 50 repeats.}
\end{figure}

Continuing the discussion in Example \ref{exa:(Symbolic-feature-mappings)} We consider a 3-dimensional input variables $(x_{1},x_{2},x_{3})\in[0,1]^{3}$
and the following model 
\begin{equation}
y=2x_{1}^{3}+5x_{3}+10+\epsilon,\label{eq:3_var_true_signal}
\end{equation}
where $\epsilon\sim N(0,\sigma^{2}).$ With sample size $n$, we generate
an $n\times3$ matrix $\bm{X}$ for input variables using uniform
random variables, and the corresponding $\bm{y}$ using \eqref{eq:3_var_true_signal}.
\global\long\def\mO{\mathcal{O}}%
We consider two architectures for generating transformations: $\mathcal{O}_{A_{u}}^{(2)}=\mO_{u}\circ\mO_{b}$
and $\mathcal{O}_{A_{b}}^{(2)}=\mO_{b}\circ\mO_{u}$, where $\mathcal{O}_{u}=\{id,x^{3}\}$
and $\mathcal{O}_{b}=\{+,\times\}$. The design matrix for each architecture
has the following dimensionality: 
\begin{enumerate}
\item For $\mathcal{O}_{A_{u}}^{(2)}$, after the first layer of binary
operations, we have $2(C_{3}^{2}+C_{3}^{1})=12$ different features
and a $n\times12$ matrix. Then, we take this $n\times12$ matrix
as the input of the next layer of unary operations and produce  $C_{2}^{1}\times12=24$
different features and a $n\times24$ matrix. 
\item For $\mathcal{O}_{A_{b}}^{(2)}$, similarly to the calculation above,
the first layer of unary operations gives $C_{2}^{1}\times3=6$ features,
and the second layer of binary operations give $2(C_{6}^{2}+C_{6}^{1})=42$
features. 
\end{enumerate}
The goal of symbolic regression is to select features from the $n\times24$
matrix (if $\mathcal{O}_{A_{u}}^{(2)}$) or $n\times42$ matrix (if
$\mathcal{O}_{A_{b}}^{(2)}$) given data.

\begin{table}
\centering %
\begin{tabular}{|>{\raggedright}m{2cm}>{\raggedright}m{2cm}|>{\raggedright}m{2cm}>{\raggedright}m{2cm}>{\raggedright}m{2cm}>{\raggedright}m{2cm}|}
\hline 
\textbf{For $\mathcal{O}_{A_{u}}^{(2)}$}  &  & \textbf{For $\mathcal{O}_{A_{b}}^{(2)}$}  &  &  & \tabularnewline
\hline 
$x_{1}+x_{1}$  & $x_{1}\times x_{3}$  & $x_{1}+x_{1}$  & $x_{1}+x_{3}$  & $x_{1}^{3}+x_{3}$  & $x_{3}+x_{3}^{3}$ \tabularnewline
$(x_{1}+x_{1})^{3}$  & $(x_{1}\times x_{3})^{3}$  & $x_{1}\times x_{1}$  & $x_{1}\times x_{3}$  & $x_{1}^{3}\times x_{3}$  & $x_{3}\times x_{3}^{3}$ \tabularnewline
$x_{1}\times x_{1}$  & $x_{3}+x_{3}$  & $x_{1}+x_{1}^{3}$  & $x_{1}+x_{3}^{3}$  & $x_{1}^{3}+x_{3}^{3}$  & $x_{3}^{3}+x_{3}^{3}$ \tabularnewline
$(x_{1}\times x_{1})^{3}$  & $(x_{3}+x_{3})^{3}$  & $x_{1}\times x_{1}^{3}$  & $x_{1}\times x_{3}^{3}$  & $x_{1}^{3}\times x_{3}^{3}$  & $x_{3}^{3}\times x_{3}^{3}$ \tabularnewline
$x_{1}+x_{3}$  & $x_{3}\times x_{3}$  & $x_{1}^{3}+x_{1}^{3}$  & $x_{1}^{3}+x_{1}^{3}$  & $x_{3}+x_{3}$  & \tabularnewline
$(x_{1}+x_{3})^{3}$  & $(x_{3}\times x_{3})^{3}$  & $x_{1}^{3}\times x_{1}^{3}$  & $x_{1}^{3}\times x_{1}^{3}$  & $x_{3}\times x_{3}$  & \tabularnewline
\hline 
\end{tabular}\caption{\label{tab:Correct-features-for}Correct features for the problem
\eqref{eq:3_var_true_signal} as they \emph{only} contains $x_{1}$,
$x_{3}$, or their transforms.}
\end{table}

The true signal in \eqref{eq:3_var_true_signal} uses only $x_{1}$
and $x_{3}$. For evaluation, we consider a feature to be ``correct''
as long as it \emph{only} contains $x_{1}$, $x_{3}$, or their transforms
(as listed out in Table \ref{tab:Correct-features-for}). To obtain a 
useful ROC curve, we first create an array called ground\_truth of
size $N_{total}$, initialized with zeros. We assign $N_{true}(=1)$
true feature to the ground truth array by setting its corresponding
element to 1. Next, we create an array called predicted\_scores of
size $N_{total}$, containing random scores for each feature. We then
assign the highest $N_{selected}$ predicted scores from our BART
selection procedure to the selected features by sorting the scores
and assigning the top $N_{selected}$ values to the selected feature
indices. We use $N_{selected}=3$ in this experiment to indicate that
there may be $x_{1},x_{3}$ and the constant intercept (which is included
by default) in \eqref{3-var-signal} will be correctly selected.

Finally, we plot the performance curve\footnote{We calculate the false positive rate (FPR) and true positive rate
(TPR) at various thresholds using the \textbf{roc\_curve} function
from \textbf{sklearn.metrics}, which takes the ground truth labels
and predicted scores as input. The AUC value is computed using the
\textbf{auc} function, which calculates the area under the ROC curve
using the trapezoidal rule.} along with the reference diagonal line representing the performance
of a random classifier. According to the criterion where we consider
a feature to be ``correct'' as long as it \emph{only} contains $x_{1}$,
$x_{3}$, we can examine each of the (24 or 42) features and label
them as 1 if ``correct''; as 0 if not. With this manually examined
ground truth label, we also compare the $N_{selected}=3$ labels to
this ground truth, we can compute the precision-recall curve and its
AUC. The AUC value is shown in the legend, providing a measure of
the performance of our feature selection method. The higher the AUC,
the better our method is at identifying the true feature among the
selected features. In symbolic regression, we focus on keeping the
correct signals involving active variables, metrics like AUC for Precision-Recall
curve is more appropriate for evaluating the performance of each method
than the usual TDR/FDR AUC, as we provided in Figure \ref{fig:PR_AUC}.
A high PR AUC indicates that the model achieves both high recall and
high precision, maintaining a good balance, especially when a positive
class is of great interest or when negative examples outnumber positive
ones.

From Figure \ref{fig:PR_AUC}, we can observe that as the noise variance
increases, the AUC decreases. It is also of interest to observe that
in the $\mathcal{O}_{A_{u}}^{(2)}$ setting, the AUC is higher than
that of the architecture of $\mathcal{O}_{A_{b}}^{(2)}$. This lends
support to the architecture design in \citet{ye2021operator} that
the binary operator should be introduced as the first alternating
layer.

\subsection{Comparison against other methods}

Furthermore, we use the same experiment to compare the performance
of different model-based feature selection methods, and our $\mathcal{T}_{0}$
using the signal \eqref{eq:3_var_true_signal}. We select the features
using $\bm{y}$ with different additive noise variances and the
corresponding $\bm{X}$ from $\mathcal{O}_{A_{u}}^{(2)}$ (24 features)
or $\mathcal{O}_{A_{b}}^{(2)}$ (42 features) architectures. For comparison,
we include LASSO ($\mathtt{glmnet}==4.1-8$, \citet{hastie2009elements}
with lambda chosen by default cross-validation (lambda=-1)), SCAD
($\mathtt{ncvreg==3.14.1}$, \citet{fan2014nonparametric}) and step-wise
subset selection using linear models (LMSTEPWISE, $\mathtt{leaps::regsubsets==3.1}$)
as competitors of model-based feature selections methods. Our $\mathcal{T}_{0}$
in \eqref{eq:T0_statistics} inspired by ranking perspective from
BART is the only method that is not model-based. In what follows,
our concern is the correct feature selection instead of predictive
performance. Thus, we simply use the same dataset for selecting the
features.

Previously, we observed in Example \ref{exa:(Power-of-)} that $\mathcal{T}_{0}$
behaves differently than classical correlation coefficient. In this
set of experimental results in Figure \ref{fig:Average_Inclusion_Probability},
we display the \emph{average inclusion probability} (AIP) as an approximation to the frequency of features being selected, since $\mathcal{T}_{0}$
is not a formal feature selection method that can be evaluated by
AUC curve, yet we still want to see how well it performs when we select
the feature with smallest concordant divergence statistics. All those
features \emph{only} contains $x_{1}$, $x_{3}$ or their transforms
(of any kind) are considered correct if selected. Then we repeat the
experiment for 50 different random $\bm{X}$ and computed the frequency
that each of these correct features are selected. Then we sum up these
probabilities and divide by $N_{selected}=3$, as our AIP metric in
Figure \ref{fig:Average_Inclusion_Probability}. That means we ask
each method to pick $N_{selected}=3$ features among all possible
features for 50 times, and AIP represents the average probability
that these features are all correct. The higher AIP means more correct
features are chosen in this configuration of sample size, noise variance, 
and method.

It is not hard to see that both BART and $\mathcal{T}_{0}$ are performing
extremely well for low noise variances, followed by LASSO. LASSO and
SCAD are behaved surprisingly well in this 2-layer example perhaps
due to the relatively small number of features, in contrast to the
\citet{ye2021operator}'s setting where a much large number of features
need to be screened. However, the low AIP associated with linear model
stepwise selection (LMSTEPWISE) is clearly not suitable for this scenario
for $\mathcal{O}_{A_{u}}^{(2)}$ nor $\mathcal{O}_{A_{b}}^{(2)}$.
One step further, we point out that $\mathcal{T}_{0}$ is the
fastest method, even if it involves summation over permutations, followed
by LASSO. While BART has the benefit of providing uncertainty quantification
and higher selection power, it is among the slower methods due to
its MCMC sampling step.\textcolor{red}{{} }

\begin{figure}[t]
\centering

\includegraphics[width=0.95\textwidth]{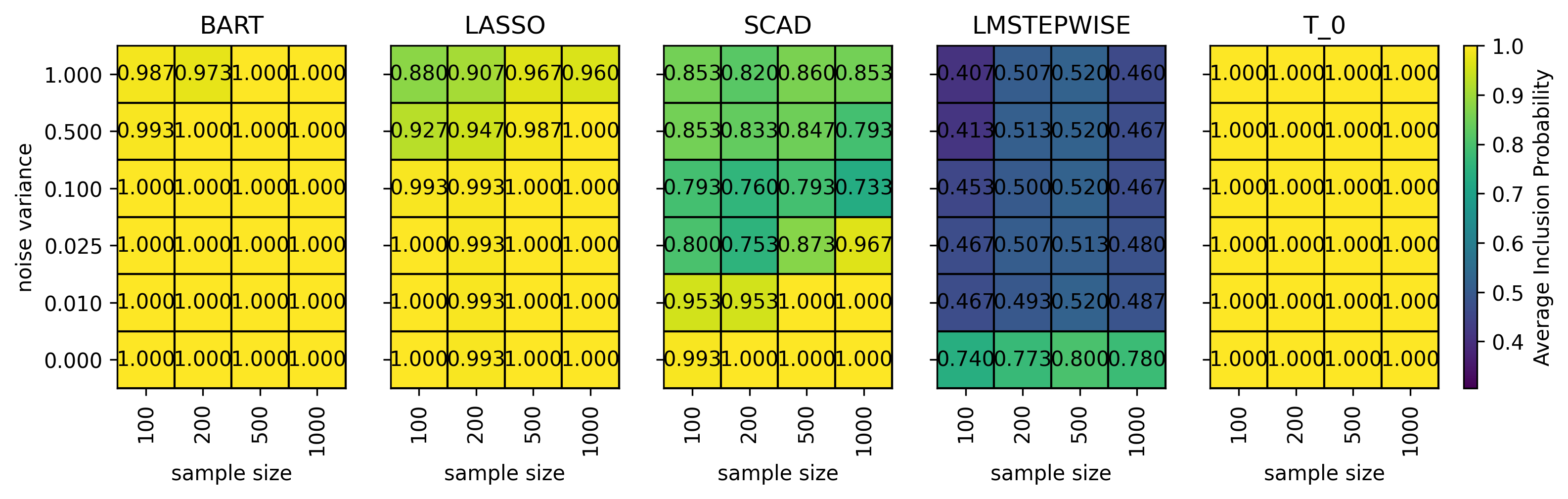}

\includegraphics[width=0.95\textwidth]{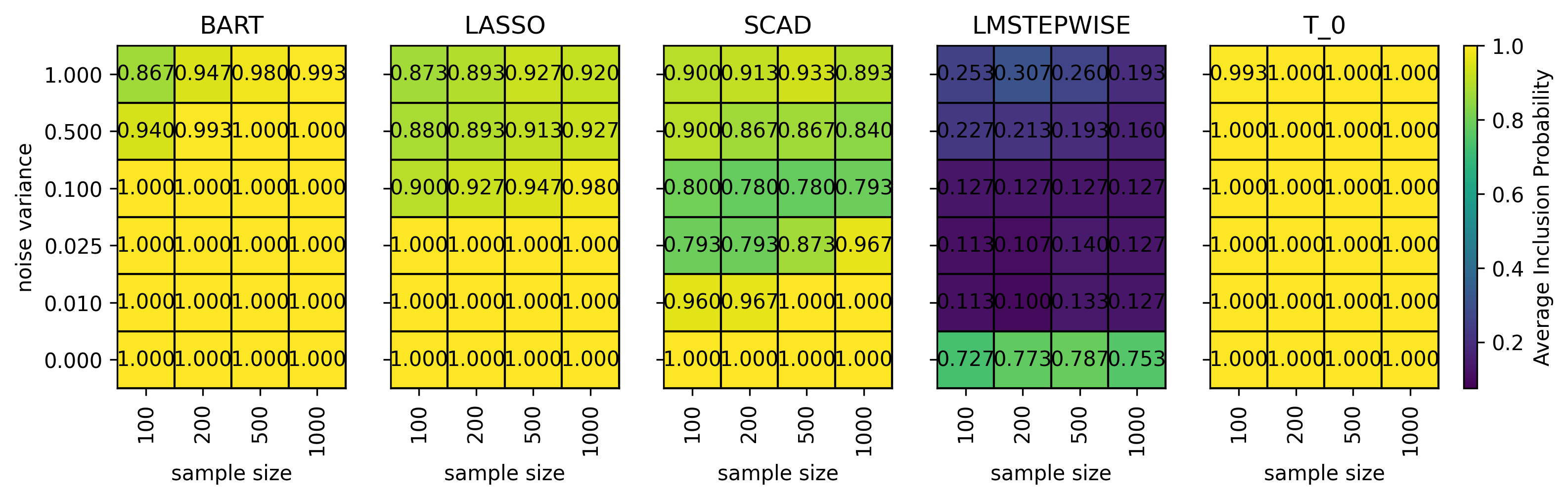}

\caption{\label{fig:Average_Inclusion_Probability} We illustrate the average
inclusion probabilities from 50 repeats of a 2-layer symbolic regression
with $\mathcal{O}_{u}=\{id,x^{3}\}$ and $\mathcal{O}_{b}=\{+,\times\}$.
The true signal is \eqref{eq:3_var_true_signal}. The first row corresponds
to the architecture of $\mathcal{O}_{A_{u}}^{(2)}$ and the second
row corresponds to the architecture of $\mathcal{O}_{A_{b}}^{(2)}$. }
\end{figure}

Concurrently, we also test these nonparametric methods on classic
ODE-Strogatz repository for symbolic regression dataset as an example
of ``ground-truth regression problems'' \citep{la2021contemporary}.
However, we did not intend to compete with the formal symbolic regression
methods but focus on the feature selection accuracy like above. In
this experiment, we use different orders of symbolic compositions
(e.g., ub, ubb) instead of alternating layers to ensure that the correct
composition of symbols can be obtained through the architecture. In
addition to different layers, we also use different sets of binary
and unitary operators to have better generality. The LMSTEPWISE cannot
work properly due to the >100 co-linear features in these examples.
Since all expressions contain $x_{1}$(x variate in raw data) and
$x_{2}$(y variate in raw data), we bring in a more stringent criteria:
all those features \emph{only} contains $x_{1}$, $x_{3}$ and their
correct transforms are considered correct.

To make the comparison fair, we enforce that the regression does not
attempt to estimate relevant coefficients for symbolic terms but look
at the selected features among all possible expressions. From Table
\ref{tab:Accuracy-rates-from}, we can observe that for simple ODEs
(vdp2), all four methods behave reasonably well. However, when
we study additive signals with different magnitudes (glider1, vdp1),
the LASSO and SCAD do not recognize the correct format. BART behaves
bad too, while $\mathcal{T}_{0}$ actually identify features that
coincide with the original ODE signal better. For complicated composition
(glider2), it seems that all methods except SCAD and $\mathcal{T}_{0}$
work pretty well, even if coefficient estimates are not allowed. 

\begin{table}[ht]
\centering

\begin{align*}
\text{Symbolic composition order}
\end{align*}

\begin{tabular}{lcccc}
\toprule 
dataset  & truth  & $\mathcal{O}_{u}$  & $\mathcal{O}_{b}$  & order\tabularnewline
\midrule
\midrule 
\textbf{glider1}  & $-0.05x_{1}^{2}-\sin(x_{2})$  & $\{\sin(x),x^{2}\}$  & $\{+,-\}$  & ub\tabularnewline
\midrule 
\textbf{glider2}  & $x_{1}-\cos(x_{2})/x_{1}$  & $\{\cos(x),\text{id}\}$  & $\{-,/\}$  & ubb\tabularnewline
\midrule 
\multirow{1}{*}{\textbf{vdp1}} & $-10/3x_{1}^{3}+10/3x_{1}+10x_{2}$  & \multirow{1}{*}{$\{x^{3}\}$} & \multirow{1}{*}{$\{+,-\}$} & \multirow{1}{*}{ubb}\tabularnewline
\midrule 
\textbf{vdp2}  & $-x_{1}/10$  & $/$  & $/$  & $/$\tabularnewline
\bottomrule
\end{tabular}

\begin{align*}
\text{Results by each method with noise variance } & 0.100
\end{align*}

\begin{tabular}{ccccc}
\toprule 
method  & \textbf{glider1}  & \textbf{glider2}  & \textbf{vdp1}  & \textbf{vdp2}\tabularnewline
\midrule
\midrule 
BART  & $x_{2}^{2}$  & $\cos(x_{2})/x_{2}-x_{1}-x_{2}$  & $x_{1}^{3}+x_{2}^{3}-x_{2}^{3}$  & $x_{2}$\tabularnewline
\midrule 
LASSO  & $\sin(x_{1})/\sin(x_{2})$  & $\cos(x_{2})/x_{2}-x_{1}-x_{2}$  & $x_{2}^{3}$  & $x_{2}$\tabularnewline
\midrule 
SCAD  & $x_{2}^{2}$  & $x_{1}$  & $x_{2}^{3}$  & $x_{2}$\tabularnewline
\midrule 
\multirow{2}{*}{$\mathcal{T}_{0}$ } & $\sin(x_{1})-\sin(x_{2})$  & $x_{1}-\cos(x_{1})/\cos(x_{2})-x_{1}$  & $x_{1}^{3}-x_{2}^{3}-x_{1}^{3}$  & $x_{2}$\tabularnewline
 & or $\sin(x_{1})-x_{1}$  &  & or $x_{2}$ &  \tabularnewline
\bottomrule
\end{tabular}

\caption{\label{tab:Accuracy-rates-from}The most frequently selected expressions
from datasets ($n=400$) in the ODE-Strogatz repository \protect\protect\protect\protect\protect\protect\protect\protect\protect\url{https://github.com/lacava/ode-strogatz},
as generated by using the first principles physical models. The LMSTEPWISE
(linear model with step-wise selection) runs into error due to the
high co-linearity in the input of these datasets.}
\end{table}

\section{\label{sec:Future-work}Discussion and Future work}

Tree-based methods are highly effective for a wide range of real-world tasks. The current understanding of this effectiveness often relies on asymptotic analysis or heuristics. While advancements in these two directions are both useful, they yield a substantial gap that calls for a formal investigation of tree-based methods that can generalize and closely link to their empirical success. In this paper, we develop a comprehensive ranking perspective for understanding tree-based
methods. We provide a series of finite-sample analyses concerning the interplay between splits and ranking, covering local splits, single trees, and tree ensembles. Asymptotics results are also established when we evaluate selected tree-based methods using their ranking performance. One particular application is symbolic feature selection in the presence of transformations of input variables, a crucial step in symbolic regression that the empirical success of tree-based methods has only been observed recently. Our ranking perspective leads to insights when comparing transformations and also provides new divergence statistics as a method to select symbolic features. 

The motivation for this work was to develop a better understanding of a broad class of tree-based methods through ranking. A future objective is to provide a foundation for more model structures to which tree-based methods can be applied, including classification, non-Gaussian error assumptions and non-standard inputs \citep{LLM2023,LHL2024}. The ranking perspective is presumably more robust to model misspecification, which might help explain the robustness of tree methods in real-world applications. Similarly, \citet{clemencon_ranking_2008} highlighted that ranking theory, when extended beyond two items, significantly depends on the designated loss, marking an interesting area for further research when the principal decision ratio $\tau$ is defined by other norms (e.g., $L^1$). Finally, our ranking perspective on tree methods can be expanded to provide uncertainty quantification for ranking.

\section*{Acknowledgment}

HL thanks for Mikael Vejdemo Johansson for kindly providing computational
resource in pilot experiments. HL was supported by the Director, Office
of Science, of the U.S. Department of Energy under Contract DE-AC02-05CH11231,
and DE-FOA-0002958. HL is also supported by NSF grant DMS 2412403. ML’s research was partially supported by NSF grant DMS/NIGMS-2153704.

\newpage
\bibliography{bsymBART}

\newpage{}

\appendix

\section*{Appendices}

\section{\label{sec:Proof-of-LemmaA}Proof of Lemma \ref{lem:LemmaA}}
\begin{proof}
Suppose that $P_{1}^{'}=\{y_{(1)}<y_{(2)}<\cdots<y_{(i+1)}\}$ and
$P_{2}^{'}=\{y_{(i)}<y_{(i+2)}<\cdots<y_{(n)}\}$ and the variance
of $P_{1}^{*}$ is strictly smaller than the variance of $P_{1}^{**}$,
then we can write explicitly that the group means for $P_{1}^{'}$
and $P_{2}^{'}$: $\mu'_{1}=\frac{1}{i}\cdot\left(\mu_{1}^{*}\cdot i-y_{(i)}+y_{(i+1)}\right)$
and $\mu'_{2}=\frac{1}{n-i}\cdot\left(\mu_{2}^{*}\cdot(n-i)-y_{(i+1)}+y_{(i)}\right)$,
where $\mu_{1}^{*},\mu_{2}^{*}$ are corresponding in-group means
of $P_{1}^{*},P_{2}^{*}$. Since we assume that the variance of $P_{1}^{*}$
is strictly smaller than the variance of $P_{1}^{**}$, our idea is
to prove that switching $y_{(i+1)}$ and $y_{(i)}$ will reduce $P_{1}^{'}$
to $P_{1}^{*}$ and $P_{2}^{'}$ to $P_{2}^{*}$ with strictly smaller
sum of group variances.

Now, we consider the differences $\mu_{1}^{*}-\mu'_{1}=\mu_{1}^{*}-\frac{1}{i}\cdot\left(\mu_{1}^{*}\cdot i-y_{(i)}+y_{(i+1)}\right)=-\frac{1}{i}\cdot\left(-y_{(i)}+y_{(i+1)}\right)<0$
and $\mu_{2}^{*}-\mu'_{2}=\mu_{2}^{*}-\frac{1}{i}\cdot\left(\mu_{2}^{*}\cdot i-y_{(i)}+y_{(i+1)}\right)=-\frac{1}{n-i}\cdot\left(-y_{(i+1)}+y_{(i)}\right)>0$.
\begin{align}
 & \sum_{y_{(j)}\in P'_{1}}(y_{(j)}-\mu_{1}^{'})^{2}+\sum_{y_{(j)}\in P'_{2}}(y_{(j)}-\mu_{2}^{'})^{2}\\
= & \sum_{j=1}^{i-1}(y_{(j)}-\mu_{1}^{'})^{2}+(y_{(i+1)}-\mu_{1}^{'})^{2}+(y_{(i)}-\mu_{2}^{'})^{2}+\sum_{j=i+2}^{n}(y_{(j)}-\mu_{\text{2}}^{'})^{2}\nonumber \\
= & \sum_{j=1}^{i-1}(y_{(j)}-\mu_{1}^{*}+\mu_{1}^{*}-\mu_{1}^{'})^{2}+(y_{(i+1)}-\mu_{2}^{*}+\mu_{2}^{*}-\mu_{1}^{'})^{2}+(y_{(i)}-\mu_{1}^{*}+\mu_{1}^{*}-\mu_{2}^{'})^{2}+\sum_{j=i+2}^{n}(y_{(j)}-\mu_{2}^{*}+\mu_{2}^{*}-\mu_{\text{2}}^{'})^{2}\nonumber \\
= & \left(\sum_{j=1}^{i-1}(y_{(j)}-\mu_{1}^{*})^{2}+\sum_{j=1}^{i-1}2(y_{(j)}-\mu_{1}^{*})(\mu_{1}^{*}-\mu_{1}^{'})+\sum_{j=1}^{i-1}(\mu_{1}^{*}-\mu_{1}^{'})^{2}\right)+(y_{(i+1)}-\mu_{2}^{*}+\mu_{2}^{*}-\mu_{1}^{'})^{2}\nonumber \\
 & +\left(\sum_{j=i+2}^{n}(y_{(j)}-\mu_{2}^{*})^{2}+\sum_{j=i+2}^{n}2(y_{(j)}-\mu_{2}^{*})(\mu_{2}^{*}-\mu_{\text{2}}^{'})+\sum_{j=i+2}^{n}(\mu_{2}^{*}-\mu_{\text{2}}^{'})^{2}\right)+(y_{(i)}-\mu_{1}^{*}+\mu_{1}^{*}-\mu_{2}^{'})^{2}\nonumber \\
= & \left({\color{blue}\sum_{j=1}^{i-1}(y_{(j)}-\mu_{1}^{*})^{2}}+2(\mu_{1}^{*}-\mu_{1}^{'})\sum_{j=1}^{i-1}(y_{(i)}-\mu_{1}^{*})+\sum_{j=1}^{i-1}(\mu_{1}^{*}-\mu_{1}^{'})^{2}\right)+{\color{red}(y_{(i+1)}-\mu_{2}^{*})^{2}}\nonumber \\
 & +2(y_{(i+1)}-\mu_{2}^{*})(\mu_{2}^{*}-\mu_{1}^{'})+(\mu_{2}^{*}-\mu_{1}^{'})^{2}\nonumber \\
 & +\left({\color{red}\sum_{j=i+2}^{n}(y_{(j)}-\mu_{2}^{*})^{2}}+2(\mu_{2}^{*}-\mu_{\text{2}}^{'})\sum_{j=1}^{i-1}(y_{(j)}-\mu_{2}^{*})+\sum_{j=i+2}^{n}(\mu_{2}^{*}-\mu_{\text{2}}^{'})^{2}\right)+{\color{blue}(y_{(i)}-\mu_{1}^{*})^{2}}\nonumber \\
 & +2(y_{(i)}-\mu_{1}^{*})(\mu_{1}^{*}-\mu_{2}^{'})+(\mu_{1}^{*}-\mu_{2}^{'})^{2}\\
= & {\color{blue}\sum_{j=1}^{i}(y_{(j)}-\mu_{1}^{*})^{2}}+2(\mu_{1}^{*}-\mu_{1}^{'})\sum_{j=1}^{i-1}(y_{(i)}-\mu_{1}^{*})+\left[(\mu_{2}^{*}-\mu_{1}^{'})^{2}+\sum_{j=1}^{i-1}(\mu_{1}^{*}-\mu_{1}^{'})^{2}\right]\nonumber \\
 & +{\color{red}\sum_{j=i+1}^{n}(y_{(j)}-\mu_{2}^{*})^{2}}+2(\mu_{2}^{*}-\mu_{\text{2}}^{'})\sum_{j=1}^{i-1}(y_{(j)}-\mu_{2}^{*})+\left[(\mu_{1}^{*}-\mu_{2}^{'})^{2}+\sum_{j=i+2}^{n}(\mu_{2}^{*}-\mu_{\text{2}}^{'})^{2}\right]\nonumber \\
\end{align}
\begin{align}
= & {\color{blue}\sum_{j=1}^{i}(y_{(j)}-\mu_{1}^{*})^{2}}-\frac{2}{i}\left(-y_{(i)}+y_{(i+1)}\right)^{2}+\left[(\mu_{2}^{*}-\mu_{1}^{'})^{2}+\sum_{j=1}^{i-1}(\mu_{1}^{*}-\mu_{1}^{'})^{2}\right]\nonumber \\
 & +{\color{red}\sum_{j=i+1}^{n}(y_{(j)}-\mu_{2}^{*})^{2}}-\frac{2}{n-i}\left(-y_{(i+1)}+y_{(i)}\right)^{2}+\left[(\mu_{1}^{*}-\mu_{2}^{'})^{2}+\sum_{j=i+2}^{n}(\mu_{2}^{*}-\mu_{\text{2}}^{'})^{2}\right]\label{eq:ref0}
\end{align}
We use red and blue colored fonts to show how we group the terms in
formula, and note that the red and blue parts are essentially the
variances of $P_{1}^{*}$ and $P_{2}^{*}$, namely $\sum_{y_{(j)}\in P_{1}^{*}}(y_{(j)}-\mu_{1}^{*})^{2}+\sum_{y_{(j)}\in P_{2}^{*}}(y_{(j)}-\mu_{2}^{*})^{2}$,
and we show below that the rest part is greater than zero. From the
assumption (for the last inequality) that $n>4,\min(n-i,i)>2$, we
have 
\begin{align}
2(\mu_{1}^{*}-\mu_{1}^{'})\sum_{j=1}^{i-1}(y_{(i)}-\mu_{1}^{*}) & =2\left(-\frac{1}{i}\cdot\left(-y_{(i)}+y_{(i+1)}\right)\right)\left(-y_{(i)}+y_{(i+1)}\right)\nonumber \\
 & =-\frac{2}{i}\left(-y_{(i)}+y_{(i+1)}\right)^{2}\geq-\left(-y_{(i+1)}+y_{(i)}\right)^{2}\label{eq:ref1}\\
2(\mu_{2}^{*}-\mu_{\text{2}}^{'})\sum_{j=1}^{i-1}(y_{(j)}-\mu_{2}^{*}) & =2\left(-\frac{1}{n-i}\cdot\left(-y_{(i+1)}+y_{(i)}\right)\right)\left(-y_{(i+1)}+y_{(i)}\right)\nonumber \\
 & =-\frac{2}{n-i}\left(-y_{(i+1)}+y_{(i)}\right)^{2}\geq-\left(-y_{(i+1)}+y_{(i)}\right)^{2}\label{eq:ref2}
\end{align}
Now we want to compare $-\frac{2}{i}\left(-y_{(i)}+y_{(i+1)}\right)^{2}$
and $(\mu_{2}^{*}-\mu_{1}^{'})^{2}$. But from the sorted assumption
and \eqref{eq:ref1}, $\mu_{1}^{'}\leq y_{(i)}<y_{(i+1)}\leq\mu_{2}^{*}$,
\begin{align}
-\frac{2}{i}\left(-y_{(i)}+y_{(i+1)}\right)^{2}+(\mu_{2}^{*}-\mu_{1}^{'})^{2} & \geq-\left(-y_{(i)}+y_{(i+1)}\right)^{2}+(\mu_{2}^{*}-\mu_{1}^{'})^{2} & \geq0\label{eq:ref3}
\end{align}
Similarly, we can compare $-\frac{2}{n-i}\left(-y_{(i+1)}+y_{(i)}\right)^{2}$
and $(\mu_{1}^{*}-\mu_{2}^{'})^{2}$ where we use \eqref{eq:ref2}
and $\mu_{1}^{*}\leq y_{(i)}<y_{(i+1)}\leq\mu_{2}^{'}$: 
\begin{align}
-\frac{2}{n-i}\left(-y_{(i+1)}+y_{(i)}\right)^{2}+(\mu_{1}^{*}-\mu_{2}^{'})^{2} & \geq-\left(-y_{(i)}+y_{(i+1)}\right)^{2}+(\mu_{2}^{*}-\mu_{1}^{'})^{2} & \geq0\label{eq:ref4}
\end{align}
Using both \eqref{eq:ref3} and \eqref{eq:ref4} in \eqref{eq:ref0},
we have proven that 
\begin{align*}
\sum_{y_{(j)}\in P_{1}^{'}}(y_{(j)}-\mu_{1}^{'})^{2}+\sum_{y_{(j)}\in P_{2}^{'}}(y_{(j)}-\mu_{2}^{'})^{2} & \geq\sum_{y_{(j)}\in P_{1}^{*}}(y_{(j)}-\mu_{1}^{*})^{2}+\sum_{y_{(j)}\in P_{2}^{*}}(y_{(j)}-\mu_{2}^{*})^{2}.
\end{align*}
This means that switching $y_{(i)}$ and $y_{(i+1)}$ indeed reduces
the total in-group variances. For more general situations, given two
partitions $P_{1},P_{2}$ of fixed sizes, and assume $\mu_{1}<\mu_{2}$.
We can first sort responses and find any pair of responses $(y_{\alpha},y_{\beta})$
such that $y_{\alpha}\in P_{1}$, $y_{\beta}\in P_{2}$ and $y_{\alpha}>y_{\beta}$.
We put $y_{\alpha}$ into $P_{2}$ and $y_{\beta}$ into $P_{1}$
and repeat the argument above to show that the in-group variances
for both partition group decreases.

Similarly, assuming that the variance of $P_{1}^{*}$ is strictly
larger than the variance of $P_{1}^{**}$, we can prove that another
global minimum of the loss function is given by assuming partitions
of $P_{1}^{**}$ and $P_{2}^{**}$. The key observation is that, the
loss can be considered as a function of two sets $P_{1}^{'}$,$P_{2}^{'}$
and there are two local minima attained by $P_{1}^{*},P_{2}^{*}$
or $P_{1}^{**},P_{2}^{**}$. The above arguments only prove that $P_{1}^{*},P_{2}^{*}$
and $P_{1}^{**},P_{2}^{**}$ both attain local minima, and they are
the only possible local minima. 
\end{proof}

\section{\label{sec:Proof-of-Proposition_93}Proof of Proposition \ref{prop:For-any-split-risk-decreases}}
\begin{proof}
Without loss of generality, we assume that the LHS takes the ordered
form $\sum_{i=1}^{n_{\text{left}}}(y_{(i)}-\mu_{L}^{C,k})^{2}+\sum_{i=n_{\text{left}}+1}^{n}(y_{(i)}-\mu_{R}^{C,k})^{2}$
where $n_{\text{left}}$ is the number of observations in the left
node. 
\begin{align*}
\sum_{i=1}^{n}(y_{(i)}-\mu^{\sharp})^{2} & =\sum_{i=1}^{n_{\text{left}}}(y_{(i)}-\mu^{\sharp})^{2}+\sum_{i=n_{\text{left}}+1}^{n}(y_{(i)}-\mu^{\sharp})^{2}\\
 & =\sum_{i=1}^{n_{\text{left}}}(y_{(i)}-\mu_{L}^{C,k}+\mu_{L}^{C,k}-\mu^{\sharp})^{2}+\sum_{i=n_{\text{left}}+1}^{n}(y_{(i)}-\mu_{R}^{C,k}+\mu_{R}^{C,k}-\mu^{\sharp})^{2}\\
 & =\sum_{i=1}^{n_{\text{left}}}(y_{(i)}-\mu_{L}^{C,k})^{2}+\sum_{i=n_{\text{left}}+1}^{n}(y_{(i)}-\mu_{R}^{C,k})^{2}+\sum_{i=1}^{n_{\text{left}}}(\mu_{L}^{C,k}-\mu^{\sharp})^{2}+\sum_{i=n_{\text{left}}+1}^{n}(\mu_{R}^{C,k}-\mu^{\sharp})^{2}\\
 & +2\underset{=0}{\underbrace{\sum_{i=1}^{n_{\text{left}}}(y_{(i)}-\mu_{L}^{C,k})(\mu_{L}^{C,k}-\mu^{\sharp})}}+2\underset{=0}{\underbrace{\sum_{i=n_{\text{left}}+1}^{n}(y_{(i)}-\mu_{R}^{C,k})(\mu_{R}^{C,k}-\mu^{\sharp})}}\\
 & =\sum_{i=1}^{n_{\text{left}}}(y_{(i)}-\mu_{L}^{C,k})^{2}+\sum_{i=n_{\text{left}}+1}^{n}(y_{(i)}-\mu_{R}^{C,k})^{2}+\sum_{i=1}^{n_{\text{left}}}(\mu_{L}^{C,k}-\mu^{\sharp})^{2}+\sum_{i=n_{\text{left}}+1}^{n}(\mu_{R}^{C,k}-\mu^{\sharp})^{2}
\end{align*}
We attained the desired inequality by dropping the third and fourth
summation in the last equality. 
\end{proof}

\section{\label{sec:Proof-prop13}Proof of Proposition \ref{prop:piecewise-monotonic-var-1}}
\begin{proof}
As the statement of the proposition, we can consider three cases as
follows, with an illustrative reference to Figure \ref{fig:Refined-monotonic-intervals}.
The corresponding three cases for (1) both of them have 0 pre-image
(2) both of them have 1 pre-image (3) one of them have 0 pre-image
and the other has 1 pre-image, are detailed as follows: 
\begin{enumerate}
\item $\theta_{1}$ has 0 pre-image of $C_{1}$ on $I$; $\theta_{2}$ has
0 pre-image of $C_{2}$ on $I$. Then, on the refined interval $I\in\mathcal{I}_{1\cap2}$,
either $\theta_{1}(u)\leq C_{1}$ or $\theta_{1}(u)>C_{1}$, for $\forall u\in I$.
Note that pre-image is well-defined when the $\theta_{1}$ (and $\theta_{2}$)
is restricted on a refined interval $I\in\mathcal{I}_{1\cap2}$. This
indicates that over $I$ the splitting value corresponding to $\theta_{1}^{-1}(C_{1})$
will not separate any $\bm{x}_{k}\in I$. Similarly, the splitting
value for $\bm{x}$ corresponding to $\theta_{2}^{-1}(C_{2})$ will
not separate any univariate inputs $x_{k}\in I$. The corresponding
terms in the principal decision ratios \ref{eq:MH_ratio_serial} becomes:
{\footnotesize
\begin{align}
\tau\mid_{I} & =\frac{\exp\left(-\sum_{i=1}^{n}(y_{i}-\mu_{L}^{1})^{2}\bm{1}(\bm{z}_{i,k_{1}}\leq C_{1})\bm{1}(\bm{x}_{i,k_{1}}\in I)-\sum_{i=1}^{n}(y_{i}-\mu_{R}^{1})^{2}\bm{1}(\bm{z}_{i,k_{1}}>C_{1})\bm{1}(\bm{x}_{i,k_{1}}\in I)\right)}{\exp\left(-\sum_{i=1}^{n}(y_{i}-\mu_{L}^{2})^{2}\bm{1}(\bm{z}_{i,k_{2}}\leq C_{2})\bm{1}(\bm{x}_{i,k_{2}}\in I)-\sum_{i=1}^{n}(y_{i}-\mu_{R}^{2})^{2}\bm{1}(\bm{z}_{i,k_{2}}>C_{2})\bm{1}(\bm{x}_{i,k_{2}}\in I)\right)}\label{eq:MH_ratio_serial-1}\\
 & =\frac{\exp\left(-\sum_{i=1}^{n_{1}}(y_{(i)}-\mu_{L}^{1})^{2}\bm{1}(\bm{z}_{(i),k_{1}}\leq C_{1})\bm{1}(\bm{x}_{(i),k_{1}}\in I)-\sum_{i=n_{1}+1}^{n}(y_{(i)}-\mu_{R}^{1})^{2}\bm{1}(\bm{z}_{(i),k_{1}}>C_{1})\bm{1}(\bm{x}_{(i),k_{1}}\in I)\right)}{\exp\left(-\sum_{i=1}^{n_{2}}(y_{(i)}-\mu_{L}^{2})^{2}\bm{1}(\bm{z}_{(i),k_{2}}\leq C_{2})\bm{1}(\bm{x}_{(i),k_{2}}\in I)-\sum_{i=n_{2}+1}^{n}(y_{(i)}-\mu_{R}^{2})^{2}\bm{1}(\bm{z}_{(i),k_{2}}>C_{2})\bm{1}(\bm{x}_{(i),k_{2}}\in I)\right)}\label{eq:MH_ratio-rank}
\end{align}
}
Since there are 0 pre-images for $\theta_{1}$ over $I$, either $\bm{1}(\bm{z}_{i,k_{1}}\leq C_{1})\bm{1}(\bm{x}_{i,k_{1}}\in I)\equiv0$
or $\bm{1}(\bm{z}_{i,k_{1}}>C_{1})\bm{1}(\bm{x}_{i,k_{1}}\in I)\equiv0$;
similarly since there are 0 pre-images for $\theta_{2}$ over $I$,
either $\bm{1}(\bm{z}_{i,k_{2}}\leq C_{2})\bm{1}(\bm{x}_{i,k_{2}}\in I)\equiv0$
or $\bm{1}(\bm{z}_{i,k_{2}}>C_{2})\bm{1}(\bm{x}_{i,k_{2}}\in I)\equiv0$.
In the case $\bm{1}(\bm{z}_{i,k_{1}}\leq C_{1})\bm{1}(\bm{x}_{i,k_{1}}\in I)\equiv0$
and $\bm{1}(\bm{z}_{i,k_{2}}\leq C_{2})\bm{1}(\bm{x}_{i,k_{2}}\in I)\equiv0$,
\ref{eq:MH_ratio-rank} becomes
\begin{align*}
\tau\mid_{I} & =\frac{\exp\left(-\sum_{i=n_{1}+1}^{n}(y_{(i)}-\mu_{R}^{1})^{2}\bm{1}(\bm{z}_{(i),k_{1}}>C_{1})\bm{1}(\bm{x}_{(i),k_{1}}\in I)\right)}{\exp\left(-\sum_{i=n_{2}+1}^{n}(y_{(i)}-\mu_{R}^{2})^{2}\bm{1}(\bm{z}_{(i),k_{2}}>C_{2})\bm{1}(\bm{x}_{(i),k_{2}}\in I)\right)}\\
 & =\frac{\exp\left(-\sum_{i=1}^{n}(y_{(i)}-\mu_{R}^{1})^{2}\right)}{\exp\left(-\sum_{i=1}^{n}(y_{(i)}-\mu_{R}^{2})^{2}\right)}\\
 & =\frac{\exp\left(-\sum_{i=1}^{n}(y_{i}-\mu_{R}^{1})^{2}\right)}{\exp\left(-\sum_{i=1}^{n}(y_{i}-\mu_{R}^{2})^{2}\right)}\\
 & =1\text{ by definition, }\mu_{R}^{1}=\mu_{R}^{2}.
\end{align*}
The case $\bm{1}(\bm{z}_{i,k_{1}}>C_{1})\bm{1}(\bm{x}_{i,k_{1}}\in I)\equiv0$
and $\bm{1}(\bm{z}_{i,k_{2}}>C_{2})\bm{1}(\bm{x}_{i,k_{2}}\in I)\equiv0$
follows the same argument. In the case $\bm{1}(\bm{z}_{i,k_{1}}\leq C_{1})\bm{1}(\bm{x}_{i,k_{1}}\in I)\equiv0$
and $\bm{1}(\bm{z}_{i,k_{2}}>C_{2})\bm{1}(\bm{x}_{i,k_{2}}\in I)\equiv0$
we have 
\begin{align*}
\tau\mid_{I} & =\frac{\exp\left(-\sum_{i=n_{1}+1}^{n}(y_{(i)}-\mu_{R}^{1})^{2}\bm{1}(\bm{z}_{(i),k_{1}}>C_{1})\bm{1}(\bm{x}_{(i),k_{1}}\in I\right)}{\exp\left(-\sum_{i=1}^{n_{2}}(y_{(i)}-\mu_{L}^{2})^{2}\bm{1}(\bm{z}_{(i),k_{2}}\leq C_{2})\bm{1}(\bm{x}_{(i),k_{2}}\in I)\right)}\\
 & =\frac{\exp\left(-\sum_{i=1}^{n}(y_{(i)}-\mu_{R}^{1})^{2}\right)}{\exp\left(-\sum_{i=1}^{n}(y_{(i)}-\mu_{L}^{2})^{2}\right)}\\
 & =\frac{\exp\left(-\sum_{i=1}^{n}(y_{i}-\mu_{R}^{1})^{2}\right)}{\exp\left(-\sum_{i=1}^{n}(y_{i}-\mu_{L}^{2})^{2}\right)}\\
 & =1\text{ by definition, }\mu_{R}^{1}=\mu_{L}^{2}.
\end{align*}
Note that this case we also have all observations allocated to right
and left nodes under two transforms. The case $\bm{1}(\bm{z}_{i,k_{1}}>C_{1})\bm{1}(\bm{x}_{i,k_{1}}\in I)\equiv0$
and $\bm{1}(\bm{z}_{(i),k_{2}}\leq C_{2})\bm{1}(\bm{x}_{i,k_{2}}\in I)\equiv0$
follows the same argument. 
\item $\theta_{1}$ has 1 pre-image of $C_{1}$ on $I$; $\theta_{2}$ has
0 pre-image of $C_{2}$ on $I$. (The discussion of the case: $\theta_{1}$
has 0 pre-image of $C_{1}$ on $I$; $\theta_{2}$ has 1 pre-image
of $C_{2}$ on $I$, is similar. ) First note that $\mu_{L}^{2}=\mu_{R}^{2}$
as $\theta_{2}$ has 0 pre-image but $\mu_{L}^{1}\neq\mu_{R}^{1}$,
yielding that the corresponding principal decision ratios: 
{\footnotesize
\begin{align*}
\tau\mid_{I} & =\frac{\exp\left(-\sum_{i=1}^{n}(y_{(i)}-\mu_{L}^{1})^{2}\bm{1}(\bm{z}_{(i),k_{1}}\leq C_{1})\bm{1}(\bm{x}_{(i),k_{1}}\in I)-\sum_{i=1}^{n}(y_{(i)}-\mu_{R}^{1})^{2}\bm{1}(\bm{z}_{(i),k_{1}}>C_{1})\bm{1}(\bm{x}_{(i),k_{1}}\in I)\right)}{\exp\left(-\sum_{i=1}^{n}(y_{(i)}-\mu_{L}^{2})^{2}\bm{1}(\bm{z}_{(i),k_{2}}\leq C_{2})\bm{1}(\bm{x}_{(i),k_{2}}\in I)\right)},\\
 & \text{ or }\frac{\exp\left(-\sum_{i=1}^{n}(y_{(i)}-\mu_{L}^{1})^{2}\bm{1}(\bm{z}_{(i),k_{1}}\leq C_{1})-\sum_{i=1}^{n}(y_{(i)}-\mu_{R}^{1})^{2}\bm{1}(\bm{z}_{(i),k_{1}}>C_{1})\bm{1}(\bm{x}_{(i),k_{1}}\in I)\right)}{\exp\left(-\sum_{i=1}^{n}(y_{(i)}-\mu_{R}^{2})^{2}\bm{1}(\bm{z}_{(i),k_{2}}>C_{2})\bm{1}(\bm{x}_{(i),k_{2}}\in I)\right)}.
\end{align*}
}
Therefore, the ratio is larger for $\theta_{1}$ by Proposition \ref{prop:For-any-split-risk-decreases}
and we should always prefer $\theta_{1}$ because the fit for $(x,y)$
using two constants $\mu_{L}^{1}\cdot\bm{1}(\bm{z}_{(i),k_{1}}\leq C_{1})$
and $\mu_{R}^{1}\cdot\bm{1}(\bm{z}_{(i),k_{1}}>C_{1})$ defined on
$I$ can not be worse than the fit for $(x,y)$ using one constant
$\mu_{L}^{2}\cdot\bm{1}(\bm{z}_{(i),k_{2}}\leq C_{2})$ (or $\mu_{R}^{2}\cdot\bm{1}(\bm{z}_{(i),k_{2}}>C_{2})$),
as stated in the next Proposition \ref{prop:For-any-split-risk-decreases}. 
\item $\theta_{1}$ has 1 pre-image of $C_{1}$ on $I$; $\theta_{2}$ has
1 pre-image of $C_{2}$ on $I$. Then, on the refined interval $I\in\mathcal{I}_{1\cap2}$,
\ref{eq:MH_ratio-rank} becomes
{\footnotesize
\begin{align*}
\tau\mid_{I} & =\frac{\exp\left(-\sum_{i=1}^{n_{1}}(y_{(i)}-\mu_{L}^{1})^{2}\bm{1}(\bm{x}_{(i),k_{1}}\leq\theta_{1}^{-1}C_{1})\bm{1}(\bm{x}_{(i),k_{1}}\in I)-\sum_{i=n_{1}+1}^{n}(y_{(i)}-\mu_{R}^{1})^{2}\bm{1}(\bm{x}_{(i),k_{1}}>\theta_{1}^{-1}C_{1})\bm{1}(\bm{x}_{(i),k_{1}}\in I)\right)}{\exp\left(-\sum_{i=1}^{n_{2}}(y_{(i)}-\mu_{L}^{2})^{2}\bm{1}(\bm{x}_{(i),k_{2}}\leq\theta_{2}^{-1}C_{2})\bm{1}(\bm{x}_{(i),k_{2}}\in I)-\sum_{i=n_{2}+1}^{n}(y_{(i)}-\mu_{R}^{2})^{2}\bm{1}(\bm{x}_{(i),k_{2}}>\theta_{2}^{-1}C_{2})\bm{1}(\bm{x}_{(i),k_{2}}\in I)\right)}
\end{align*}
}
From Corollary \ref{cor:(Optimal-2-partition-fixed-size}, we know
that both the numerator and denominator of $\tau\mid_{I}$ are fixed
size oracle 2-partitions on $y$. However, we need to decide which
of $\bm{x}_{(i),k_{1}}\leq\theta_{1}^{-1}C_{1}$ and $\bm{x}_{(i),k_{2}}\leq\theta_{2}^{-1}C_{2}$
gives us a larger sum of variances. From Lemma \ref{lem:(Swaps-to-optimize},
we know that it reduces to compare $n-1$ possible values of sum of
variances (corresponding to $n-1$ different split values). 
\end{enumerate}
\end{proof}

\section{\label{sec:Proof-of-Theorem-Oracle}Proof of Theorem \ref{thm:Oracle}}
\begin{proof}
Using Theorem 3 in \citet{cossock2006subset}, we know that for the
full dataset $\mathcal{X}^{(N)}$, the following holds: 
\begin{align}
\bm{T}\left(r_{B}\right)-\bm{T}\left(r_{c,K}\right) & \leq\frac{4}{\sqrt{N}}\left(\sum_{j=1}^{N}\left(f_{B}(\bm{x}_{j})-f_{c,K}(\bm{x}_{j})\right)^{2}\right)^{1/2},\label{eq:cossock_zhang_bound-1}
\end{align}
where $r_{c,K}$ is induced by the $f_{c,K}\in\mathcal{G}\subset\mathcal{G}_0$ constructed
from a CART. This inequality suggests that the a scoring function $f_{c,K}$
approximated by CART of depth $K\geq1$ with low approximating $L_{2}$
error to the Bayes scoring function $f_{B}$ can attain low approximating
ranking error $\bm{T}$ as well.

Assuming that the true signal $f_{B} \in \mathcal{G}_0$ and the CART-induced $f_{c,K} \in \mathcal{G} \subset \mathcal{G}_0$ as in the statement of Theorem \ref{thm:Oracle}, now we apply Theorem 4.3 in \citet{klusowski2021universal}.
We can assert that the CART prediction $f_{c,K}$ constructed from
splitting a complete binary tree using CART loss \eqref{eq:loss.y.ranking}
of depth $K$ satisfy the following universal consistency as: 
\begin{align}
\mathbb{E}_{\left(\mathcal{X}^{(N)},\mathcal{Y}^{(N)}\right)}\left(\left\Vert f_{B}-f_{c,K}\right\Vert ^{2}\right) & \leq2\inf_{g\in\mathcal{G}}\left\{ \left\Vert f_{B}-g\right\Vert ^{2}+\frac{\left\Vert g\right\Vert _{\text{TV}}^{2}}{K+3}+C_{B}\frac{2^{K}\log(Nd)}{N}\right\}, %
\end{align}
where the constant $C_{B}$ depends on the uniform bound on the total
variations along coordinates $f_{B,i}:\mathbb{R}\rightarrow\mathbb{R}$
as assumed. Applying Markov's inequality yields 
\begin{align}
\mathbb{P}_{\left(\mathcal{X}^{(N)},\mathcal{Y}^{(N)}\right)}& \left(\left\Vert f_{B}-f_{c,K}\right\Vert ^{2}>\alpha_{1}\right) \leq\frac{\mathbb{E}_{\left(\mathcal{X}^{(N)},\mathcal{Y}^{(N)}\right)}\left(\left\Vert f_{B}-f_{c,K}\right\Vert ^{2}\right)}{\alpha_{1}}\\
& \leq\frac{2}{\alpha_{1}}\inf_{g\in\mathcal{G}}\left\{ \left\Vert f_{B}-g\right\Vert ^{2}+\frac{\left\Vert g\right\Vert _{\text{TV}}^{2}}{K+3}+C_{B}\frac{2^{K}\log^{2}N\log(Nd)}{N}\right\},  \label{eq:P_sqnorm}\end{align}
for any $\alpha_{1}>0$. This means that the tree-like functions in class $\mathcal{G}$ can approximate the Bayes score function well enough, and bounded from above.  

Since the result in
\eqref{eq:P_sqnorm} considers the exact norm, we need one more step
to connect the exact norm to the empirical norm used in the statement
of Theorem 3 in \citet{cossock2006subset}. To attain this, we invoke classical results 
from empirical process regarding the convergence of empirical norm
\citep{ledoux1991probability}. According to Theorem 2.2 in \citet{van2014uniform}, the empirical norm $\|f\|_{N}^{2}\coloneqq\frac{1}{N}\sum_{i=1}^{N}f(\bm{x}_{i})^{2}$
converges to the exact $\ell^{2}$-norm $\|f\|^{2} = \int f(u)^{2}d\mathbb{P}_X(u)$
at a rate of $\mathcal{O}\left(\frac{1}{\sqrt{N}}\right)$ for any
$f\in\mathcal{G}_0$ from a possibly larger class than tree-like functions $\mathcal{G}$. In particular, we have that the empirical norm of $f\in\mathcal{G}_0$ can be approximated as well
\begin{align}
\mathbb{E}_{\left(\mathcal{X}^{(N)},\mathcal{Y}^{(N)}\right)}\left(\sup_{f\in\mathcal{G}_0}\left|\|f\|_{N}-\|f\|\right|\right) & \leq\frac{2R_{\infty}J_{0}(2R_{2},\mathcal{G}_0)}{\sqrt{N}},
\end{align} 
which implies 
\begin{align}
\mathbb{P}_{\left(\mathcal{X}^{(N)},\mathcal{Y}^{(N)}\right)}\left(\sup_{f\in\mathcal{G}_0}\left|\|f\|_{N}-\|f\|\right|>\alpha_{2}\right) & \leq\frac{\mathbb{E}_{\left(\mathcal{X}^{(N)},\mathcal{Y}^{(N)}\right)}\left(\sup_{f\in\mathcal{G}_0}\left|\|f\|_{N}-\|f\|\right|\right)}{\alpha_{2}}\\
 & \leq\frac{2R_{\infty}J_{0}(2R_{2},\mathcal{G}_0)}{\alpha_{2}\sqrt{N}},\label{eq:P_supnorm}
\end{align}
for any $\alpha_{2}>0$ by applying Markov's inequality. 
Applying the inequality in \eqref{eq:cossock_zhang_bound-1}, we obtain 
\begin{align}
 & \mathbb{P}_{\left(\mathcal{X}^{(N)},\mathcal{Y}^{(N)}\right)}\left(\bm{T}\left(r_{B}\right)-\bm{T}\left(r_{c,K}\right)>\varepsilon_{N}\right)\nonumber \\
 & \leq\mathbb{P}_{\left(\mathcal{X}^{(N)},\mathcal{Y}^{(N)}\right)}\left(\frac{4}{\sqrt{N}}\left(\sum_{j=1}^{N}\left(f_{B}(\bm{x}_{j})-f_{c,K}(\bm{x}_{j})\right)^{2}\right)^{1/2}>\varepsilon_{N}\right)\nonumber \\
 & =\mathbb{P}_{\left(\mathcal{X}^{(N)},\mathcal{Y}^{(N)}\right)}\left(\frac{1}{N}\left(\sum_{j=1}^{N}\left(f_{B}(\bm{x}_{j})-f_{c,K}(\bm{x}_{j})\right)^{2}\right)>\frac{\varepsilon_{N}^{2}}{16}\right)\nonumber \\
 & =\mathbb{P}_{\left(\mathcal{X}^{(N)},\mathcal{Y}^{(N)}\right)}\left(\left\Vert f_{c,K}-f_{B}\right\Vert _{N}>\frac{\varepsilon_{N}^{2}}{16}\text{ and }\left|\left\Vert f_{c,K}-f_{B}\right\Vert _{N}-\left\Vert f_{c,K}-f_{B}\right\Vert \right|\leq\frac{\varepsilon_{N}^{2}}{32}\right)\nonumber \\
 & \quad +\mathbb{P}_{\left(\mathcal{X}^{(N)},\mathcal{Y}^{(N)}\right)}\left(\left\Vert f_{c,K}-f_{B}\right\Vert _{N}>\frac{\varepsilon_{N}^{2}}{16}\text{ and }\left|\left\Vert f_{c,K}-f_{B}\right\Vert _{N}-\left\Vert f_{c,K}-f_{B}\right\Vert \right|>\frac{\varepsilon_{N}^{2}}{32}\right)\nonumber \\
 & =: A_1 + A_2. 
\end{align}
For the first term $A_1$, it represents how well the tree-like function class $\mathcal{G}$ can approximate the Bayes scoring functions. We substitute $\alpha_{1}=\frac{\varepsilon_{N}^{2}}{32}$ into \eqref{eq:P_sqnorm} and bound it by 
\[
\mathbb{P}_{\left(\mathcal{X}^{(N)},\mathcal{Y}^{(N)}\right)}\left(\left\Vert f_{B}-f_{c,K}\right\Vert ^{2}>\left[\frac{\varepsilon_{N}^{2}}{32}\right]^{2}\right)\leq\frac{1024}{\varepsilon_{N}^{4}}\cdot\inf_{g\in\mathcal{G}}\left\{ \left\Vert f_{B}-g\right\Vert ^{2}+\frac{\left\Vert g\right\Vert _{\text{TV}}^{2}}{K+3}+C_{B}\frac{2^{K}\log^{2}N\log(Nd)}{N}\right\}. 
\]
For the second term $A_2$, it represents how fast empirical norm with respect to $\bm{X}$ converges when the Bayesian scoring function is in a potentially larger class $\mathcal{G}_0$ (but the tree approximating scoring function $f_{c,K}\in\mathcal{G}$).  This enclosed event has an intersection component in form of \eqref{eq:P_supnorm} with $\alpha_{2}=\frac{\varepsilon_{N}^{2}}{16}$. Therefore, $A_2 \leq \frac{32R_{\infty}J_{0}(2R_{2},\mathcal{G}_0)}{\varepsilon_{N}^{2}\sqrt{N}}.$ 

Combining these two bounds yields  
\begin{align*}
& \mathbb{P}_{\left(\mathcal{X}^{(N)},\mathcal{Y}^{(N)}\right)}\left(\bm{T}\left(r_{B}\right)-\bm{T}\left(r_{c,K}\right)>\varepsilon_{N}\right) \\ & \leq\frac{1024}{\varepsilon_{N}^{4}}\cdot\inf_{g\in\mathcal{G}}\left\{ \left\Vert f_{B}-g\right\Vert ^{2}+\frac{\left\Vert g\right\Vert _{\text{TV}}^{2}}{K+3}+C_{B}\frac{2^{K}\log^{2}N\log(Nd)}{N}\right\} +\frac{32R_{\infty}J_{0}(2R_{2},\mathcal{G}_0)}{\varepsilon_{N}^{2}\sqrt{N}}.
\end{align*}
This completes the proof of the inequality~\eqref{eq:rank.bound.CART.two.class}. 

When $\mathcal{G}_0=\mathcal{G}$, substituting $g = f_B$ into the upper bound~\eqref{eq:rank.bound.CART.two.class} leads to the simplified upper bound~\eqref{eq:oracle.special.case}. 
\end{proof}

\section{\label{sec:Proof-of-Asymptotics}Proof of Theorem \ref{thm:Asymptotics}}
\begin{proof}
Using Theorem 3 in \citet{cossock2006subset}, we know that for the
full dataset $\mathcal{X}^{(N)}$ we have the bound 
\begin{align}
\bm{T}\left(r_{B}\right)-\bm{T}\left(r_{f}\right) & \leq\frac{4}{\sqrt{N}}\left(\sum_{j=1}^{N}\left(f_{B}(\bm{x}_{j})-f(\bm{x}_{j})\right)^{2}\right)^{1/2}\label{eq:cossock_zhang_bound}
\end{align}
Therefore, it justifies that the a scoring function $f$ with low
approximating $L_{2}$ error to the Bayes scoring function $f_{B}$
can attain low approximating ranking error $\bm{T}$ as well.

\citet{rovckova2019theory} presents the posterior concentration of
BART in their Theorem 7.1 (adapted to our notations above). Namely,
when the $f$ is $\nu$-Holder continuous with $0<\nu\leq1$ and $\|f_{B}\|_{\infty}\apprle\log^{1/2}N$,
assume a regular design $\mathcal{X}^{(N)}=\{\bm{x}_{1},\bm{x}_{2},\cdots,\bm{x}_{N}\}\subset\mathbb{R}^{d},d\lesssim\log^{1/2}N$
for the features. With a fixed number of trees and node $\eta$ splitting
probability $p_{\text{split}}(\eta)=\alpha^{\text{depth}(\eta)},\alpha\in\left[\frac{1}{N},\frac{1}{2}\right)$
proportional to the depth of node $\eta$ with respect to each tree,
we have posterior concentration results from BART, when the regression
aims at approximating the optimal $f_{B}$. Precisely we assert that
the BART posterior measure $\Pi(\cdot\mid y_{1},y_{2},\cdots,y_{N})$
concentrates on all scoring functions $f$ that is measurable with
respect to the product $\sigma$-field $\mathcal{F}$ generated by the joint measure of 
$y_{1},y_{2},\cdots,y_{N}$. That is, when the approximating scoring function
$f$ is approximated by BART, we have 
\begin{align*}
\prod\left(\left.f\in\mathcal{F}:\left\Vert f(\bm{x})-f_{B}(\bm{x})\right\Vert _{N}>M_{N}\varepsilon_{N}\right|y_{1},y_{2},\cdots,y_{N}\right) & \rightarrow0,\text{as }N\rightarrow\infty\text{ for all }\bm{x}\in\mathcal{X}^{(N)}
\end{align*}
for any sequence $M_{N}\rightarrow0$ in the joint probability of $y_{1},y_{2},\cdots,y_{N}$,
as the sample size $N$ and the dimensionality $d\rightarrow\infty$
and $\varepsilon_{N}=N^{-\alpha/(2\alpha+d)}\log^{1/2}N$. Here we
use the empirical norm definition $\|f\|_{N}^{2}\coloneqq\frac{1}{N}\sum_{i=1}^{N}f(\bm{x}_{i})^{2}$
and \eqref{eq:cossock_zhang_bound} in the second line. 
\begin{align*}
 & \prod\left(\left.f\in\mathcal{F}:\bm{T}\left(r_{B}\right)-\bm{T}\left(r_{f}\right)>M_{N}\varepsilon_{N}\right|y_{1},y_{2},\cdots,y_{N}\right)\\
 & \leq\prod\left(\left.f\in\mathcal{F}:\frac{4}{\sqrt{N}}\left(\sum_{j=1}^{N}\left(f_{B}(\bm{x}_{j})-f(\bm{x}_{j})\right)^{2}\right)^{1/2}>M_{N}\varepsilon_{N}\right|y_{1},y_{2},\cdots,y_{N}\right)\\
 & =\prod\left(\left.f\in\mathcal{F}:\frac{1}{N}\left(\sum_{j=1}^{N}\left(f_{B}(\bm{x}_{j})-f(\bm{x}_{j})\right)^{2}\right)>\frac{M_{N}^{2}\varepsilon_{N}^{2}}{16}\right|y_{1},y_{2},\cdots,y_{N}\right)\\
 & =\prod\left(\left.f\in\mathcal{F}:\left\Vert f-f_{B}\right\Vert _{N}>M_{N}^{\sharp}\varepsilon_{N}\right|y_{1},y_{2},\cdots,y_{N}\right)\rightarrow0,\text{as the design size }N\rightarrow\infty.
\end{align*}
where $M_{N}^{\sharp}=\frac{M_{N}^{2}\varepsilon_{N}}{16}$ can be
chosen to be any sequence converging to 0 as $N\rightarrow\infty$.
The key of this argument is to convert the error measured in $\bm{T}$
metric to the empirical norm. 
\end{proof}

\section{\label{sec:Proof-of-LemmaA1}Proof of Lemma \ref{lem:(Swaps-to-optimize}}
\begin{proof}
Either swapping $(y_{\alpha},y_{\gamma})$ or $(y_{\beta},y_{\gamma})$
leaves us with no reversed pairs. Thus, it suffices to note that if
we swap $(y_{\alpha},y_{\gamma})$ the RHS of \eqref{eq:ref3} and
\eqref{eq:ref4} become 
\begin{align*}
-\left(-y_{\gamma}+y_{\alpha}\right)^{2}+(\mu_{2,(\alpha,\gamma)}^{*}-\mu_{1}^{'})^{2} & \geq0,
\end{align*}
where the $\mu_{1}^{'}=\frac{1}{n_{1}}\left(y_{\alpha}+y_{\beta}+\sum_{y\in P_{1},y\neq y_{\alpha},y_{\beta}}y\right)$
and $\mu_{2,(\alpha,\gamma)}^{*}=\frac{1}{n_{2}}\left(y_{\alpha}+\sum_{y\in P_{2},y\neq y_{\gamma}}y\right)$.
Similarly, if we swap $(y_{\beta},y_{\gamma})$ the RHS of \eqref{eq:ref3}
and \eqref{eq:ref4} become 
\begin{align*}
-\left(-y_{\gamma}+y_{\beta}\right)^{2}+(\mu_{2,(\beta,\gamma)}^{*}-\mu_{1}^{'})^{2} & \geq0,
\end{align*}
where the $\mu_{1}^{'}=\frac{1}{n_{1}}\left(y_{\alpha}+y_{\beta}+\sum_{y\in P_{1},y\neq y_{\alpha},y_{\beta}}y\right)$
and $\mu_{2,(\beta,\gamma)}^{*}=\frac{1}{n_{2}}\left(y_{\beta}+\sum_{y\in P_{2},y\neq y_{\gamma}}y\right)$.
It follows that $\mu_{2,(\alpha,\gamma)}^{*}>\mu_{2,(\beta,\gamma)}^{*}>\mu_{1}^{'}$,
and the assumption $y_{\alpha}>y_{\beta}>y_{\gamma}$ that 
\begin{align*}
 & -\left(-y_{\gamma}+y_{\alpha}\right)^{2}+(\mu_{2,(\alpha,\gamma)}^{*}-\mu_{1}^{'})^{2}+\left(-y_{\gamma}+y_{\beta}\right)^{2}-(\mu_{2,(\beta,\gamma)}^{*}-\mu_{1}^{'})^{2}\\
= & \left[\left(-y_{\gamma}+y_{\beta}\right)^{2}-\left(-y_{\gamma}+y_{\alpha}\right)^{2}\right]+\left[(\mu_{2,(\alpha,\gamma)}^{*}-\mu_{1}^{'})^{2}-(\mu_{2,(\beta,\gamma)}^{*}-\mu_{1}^{'})^{2}\right]\geq0.
\end{align*}
This means that the reduction \eqref{eq:ref1}$+$\eqref{eq:ref2}
is larger if we swap $(y_{\alpha},y_{\gamma})$. 
\end{proof}

\section{\label{sec:Proof-of-Proposition-T0}Proof of Proposition \ref{prop:-Conditioned-on}}
\begin{proof}
 The statistics $\mathcal{T}_{0}(g)$ in \eqref{eq:T0_statistics}
can be written as 
\[
\sum_{\pi}\left\{ \bm{1}\left(\bm{x}_{\pi(1)}\geq\bm{x}_{\pi(2)}\right)\cdot\left|y_{\pi(1)}-y_{\pi(2)}\right|\cdot\bm{1}(y_{\pi(1)}<y_{\pi(2)})+\bm{1}\left(\bm{x}_{\pi(1)}<\bm{x}_{\pi(2)}\right)\cdot\left|y_{\pi(1)}-y_{\pi(2)}\right|\cdot\bm{1}(y_{\pi(1)}\geq y_{\pi(2)})\right\} .
\]
Without loss of generality, we next consider one summation involving $\bm{1}\left(f(\bm{x}_{\pi(1)})\geq f(\bm{x}_{\pi(2)})\right)$,
and the argument remains the same for the other summation. 
Now since $\bm{X}_{k},k\in\{1,\cdots,d\}$ is inactive, the distribution
of $y$ is independent of the distribution of $f(\bm{X}_{k})$. Based
on this observation, we can remove the conditioning inside the expectation with respect to $\bm{X}_k$ in the first line: 
\begin{align*}
\mathbb{E}_{\bm{X}_{k},\bm{y}}\mathcal{T}_{0}(g) & =\frac{2}{n(n-1)}\cdot\left[\mathbb{P}_{\bm{X}_{k}}\left(f(\bm{x}_{\pi(1)})\geq f(\bm{x}_{\pi(2)})\right)\right]\cdot\sum_{\pi}\left[\mathbb{E}_{\bm{y}\mid\bm{X}_{k}}\left|y_{\pi(1)}-y_{\pi(2)}\right|\cdot\bm{1}(y_{\pi(1)}<y_{\pi(2)})\right]\\
 & =\frac{2}{n(n-1)}\cdot\left[\mathbb{P}_{\bm{X}_{k}}\left(f(\bm{x}_{\pi(1)})\geq f(\bm{x}_{\pi(2)})\right)\right]\cdot\sum_{\pi}\left[\mathbb{E}_{\bm{y}}\left|y_{\pi(1)}-y_{\pi(2)}\right|\cdot\bm{1}(y_{\pi(1)}<y_{\pi(2)})\right]\\
 & =\left[\mathbb{P}_{\bm{X}_{k}}\left(f(\bm{x}_{\pi(1)})\geq f(\bm{x}_{\pi(2)})\right)\right]\cdot\frac{2}{n(n-1)}\cdot\frac{1}{2}\sum_{\pi}\left[\mathbb{E}_{\bm{y}}\left|y_{\pi(1)}-y_{\pi(2)}\right|\right]\\
 & \asymp O\left(1\right)\cdot\left[\mathbb{P}_{\bm{X}_{k}}\left(f(\bm{x}_{\pi(1)})\geq f(\bm{x}_{\pi(2)})\right)\right]>0,
\end{align*}
where the second-to-last line uses an symmetric argument as 
\[
\sum_{\pi}\left[\mathbb{E}_{\bm{y}}\left|y_{\pi(1)}-y_{\pi(2)}\right|\cdot\bm{1}(y_{\pi(1)}<y_{\pi(2)})\right] = \sum_{\pi}\left[\mathbb{E}_{\bm{y}}\left|y_{\pi(1)}-y_{\pi(2)}\right|\cdot\bm{1}(y_{\pi(1)}> y_{\pi(2)})\right], 
\]
and their sum is $\sum_{\pi}\left[\mathbb{E}_{\bm{y}}\left|y_{\pi(1)}-y_{\pi(2)}\right|\right].$
This
proves part (i).

On the other hand, if we  take the expectation with respect to all $\bm{y}\mid\bm{X}$:
\begin{align}
\mathbb{E}_{\bm{X},\bm{y}}\mathcal{T}_{0}(g) & =\frac{2}{n(n-1)}\cdot\mathbb{E}_{\bm{X},\bm{y}}\sum_{\pi}\bm{1}\left(f(\bm{x}_{\pi(1)})\geq f(\bm{x}_{\pi(2)})\right)\cdot\left|y_{\pi(1)}-y_{\pi(2)}\right|\cdot\bm{1}(y_{\pi(1)}<y_{\pi(2)})\nonumber \\
 & =\frac{2}{n(n-1)}\cdot\mathbb{E}_{\bm{X}}\mathbb{E}_{\bm{y}\mid\bm{X}}\sum_{\pi}\bm{1}\left(f(\bm{x}_{\pi(1)})\geq f(\bm{x}_{\pi(2)})\right)\cdot\left|y_{\pi(1)}-y_{\pi(2)}\right|\cdot\bm{1}(y_{\pi(1)}<y_{\pi(2)})\nonumber \\
 & =\frac{2}{n(n-1)}\cdot\left[\mathbb{E}_{\bm{X}}\bm{1}\left(f(\bm{x}_{\pi(1)})\geq f(\bm{x}_{\pi(2)})\right)\right]\cdot\sum_{\pi}\left[\mathbb{E}_{\bm{y}\mid\bm{X}}\left|y_{\pi(1)}-y_{\pi(2)}\right|\cdot\bm{1}(y_{\pi(1)}<y_{\pi(2)})\right]\nonumber \\
 & =\frac{2}{n(n-1)}\cdot\left[\mathbb{P}_{\bm{X}}\left(f(\bm{x}_{\pi(1)})\geq f(\bm{x}_{\pi(2)})\right)\right]\cdot\sum_{\pi}\left[\mathbb{E}_{\bm{y}\mid\bm{X}}\left|y_{\pi(1)}-y_{\pi(2)}\right|\cdot\bm{1}(y_{\pi(1)}<y_{\pi(2)})\right].\label{eq:inter-step_inactive}
\end{align}
if there exists such a $g$ that $g(\bm{x}_{1})\geq g(\bm{x}_{2})\Leftrightarrow$$y_{1}\geq y_{2}$
then 
\begin{align*}
\eqref{eq:inter-step_inactive} & =\left[\mathbb{P}_{\bm{X}}\left(g(\bm{x}_{\pi(1)})\geq g(\bm{x}_{\pi(2)})\right)\right]\cdot\sum_{\pi}\frac{2}{n(n-1)}\cdot\left[\mathbb{E}_{\bm{y}\mid\bm{X}}\left|y_{\pi(1)}-y_{\pi(2)}\right|\cdot\bm{1}\left(g(\bm{x}_{\pi(1)})<g(\bm{x}_{\pi(2)})\right)\right]\\
 & =\left[\mathbb{P}_{\bm{X}}\left(g(\bm{x}_{\pi(1)})\geq g(\bm{x}_{\pi(2)})\right)\right]\cdot\sum_{\pi}\frac{2}{n(n-1)}\cdot\left[\mathbb{E}_{\bm{y}\mid\bm{X}}\left(y_{\pi(2)}-y_{\pi(1)}\right)\cdot\bm{1}\left(g(\bm{x}_{\pi(1)})<g(\bm{x}_{\pi(2)})\right)\right]\\
 & =\left[\mathbb{P}_{\bm{X}}\left(g(\bm{x}_{\pi(1)})\geq g(\bm{x}_{\pi(2)})\right)\right]\cdot\frac{2}{n(n-1)}\\
 & \times\left\{ \sum_{\pi}\mathbb{E}_{\bm{y}\mid\bm{X}}y_{\pi(2)}\cdot\bm{1}\left(g(\bm{x}_{\pi(1)})<g(\bm{x}_{\pi(2)})\right)-\sum_{\pi}\mathbb{E}_{\bm{y}\mid\bm{X}}y_{\pi(1)}\cdot\bm{1}\left(g(\bm{x}_{\pi(1)})<g(\bm{x}_{\pi(2)})\right)\right\} 
\end{align*}
For the other summation we have equal value $\sum_{\pi}\mathbb{E}_{\bm{y}\mid\bm{X}}y_{\pi(2)}\cdot\bm{1}\left(g(\bm{x}_{\pi(1)})\geq g(\bm{x}_{\pi(2)})\right)-\sum_{\pi}\mathbb{E}_{\bm{y}\mid\bm{X}}y_{\pi(1)}\cdot\bm{1}\left(g(\bm{x}_{\pi(1)})\geq g(\bm{x}_{\pi(2)})\right)$
and cancels out the last row of expressions. This proves part (ii). 
\end{proof}

\end{document}